\documentclass{amsart}

\usepackage{hyperref}
\hypersetup{
    colorlinks=true,
    linkcolor=blue,
    citecolor=blue,
    urlcolor=blue,
}

\makeatletter
\newcommand\org@maketitle{}
\newcommand\@authors{}
\let\org@maketitle\maketitle
\def\maketitle{%
	\let\@authors\authors
	\nxandlist{; }{ and }{; }\@authors
	\hypersetup{
		linktocpage=true,
		pdftitle={\@title},
                pdfauthor={\@authors},
                pdfsubject={\subjclassname. \@subjclass},
		pdfkeywords={\@keywords}
	}%
	\org@maketitle
}
\makeatother

\usepackage{amssymb, accents}
\usepackage{rsfso}
\usepackage{mathtools}
\usepackage{microtype}
\usepackage[alphabetic,msc-links]{amsrefs}
\usepackage{enumitem}
\usepackage{amsfonts}
\DeclareMathAlphabet{\mathcal}{OMS}{cmsy}{m}{n}

\usepackage{doi}
\renewcommand{\PrintDOI}[1]{\doi{#1}}

\numberwithin{equation}{section}

\newtheorem{theorem}{Theorem}[section]
\newtheorem{lemma}[theorem]{Lemma}

\theoremstyle{definition}
\newtheorem{definition}[theorem]{Definition}

\theoremstyle{remark}
\newtheorem{remark}[theorem]{Remark}

\newcommand{\cH}{{\mathcal H}}
\newcommand{\cD}{{\mathcal D}}
\newcommand{\cQ}{{\mathcal Q}}

\newcommand{\al}{\alpha}
\newcommand{\be}{\beta}
\newcommand{\g}{\gamma}
\newcommand{\de}{\delta}

\newcommand{\la}{\lambda}
\newcommand{\ka}{\kappa}

\newcommand{\R}{\mathbb{R}}
\newcommand{\vp}{\varphi}

\newcommand{\bb}{\mathbf{b}}
\newcommand{\bg}{\mathbf{g}}
\newcommand{\bh}{\mathbf{h}}
\newcommand{\bff}{\mathbf{f}}
\newcommand{\bq}{\mathbf{q}}

\newcommand{\bbQ}{\mathbb{Q}}

\newcommand{\D}{\nabla}

\newcommand{\supp}{\operatorname{supp}}
\renewcommand{\div}{\operatorname{div}}

\newcommand{\dist}{\operatorname{dist}}

\newcommand{\mean}[1]{\langle #1\rangle}

\newcommand{\DMO}{\operatorname{DMO}}
\newcommand{\tr}{\operatorname{tr}}

\def\Xint#1{\mathchoice
  {\XXint\displaystyle\textstyle{#1}}%
  {\XXint\textstyle\scriptstyle{#1}}%
  {\XXint\scriptstyle\scriptscriptstyle{#1}}%
  {\XXint\scriptscriptstyle\scriptscriptstyle{#1}}%
  \!\int}
\def\XXint#1#2#3{{\setbox0=\hbox{$#1{#2#3}{\int}$}
    \vcenter{\hbox{$#2#3$}}\kern-.5\wd0}}

\def\dashint{\Xint-}

\mathchardef\ordinarycolon\mathcode`\:
\mathcode`\:=\string"8000
\begingroup \catcode`\:=\active
  \gdef:{\mathrel{\mathop\ordinarycolon}}
\endgroup

\author{Hongjie Dong}
\author{Seongmin Jeon}

\address[H. Dong]{Division of Applied Mathematics
\newline\indent
Brown University
\newline\indent
182 George Street, Providence RI 02912, USA}
\email{hongjie\_dong@brown.edu}
\address[S. Jeon]{Department of Mathematics Education
\newline\indent
Hanyang University
\newline\indent
222 Wangsimni-ro, Seongdong-gu, Seoul 04763, Republic of Korea}
\email{seongminjeon@hanyang.ac.kr}
\thanks{H. Dong is partially supported by the NSF under agreement DMS-2350129. S. Jeon was supported by the Academy of Finland grant 347550, the research fund of Hanyang University(HY-202400000003278), and the National Research Foundation of Korea(NRF) grant 
funded by the Korea government(MSIT) (RS-2025-24803159).}

\title[Schauder type estimates for degenerate or singular systems]{Schauder type estimates for degenerate or singular parabolic systems with partially DMO coefficients}

\allowdisplaybreaks

\begin{document}
\begin{abstract}
We study elliptic and parabolic systems in divergence form with degenerate or singular coefficients. Under the conormal boundary condition on the flat boundary, we establish boundary Schauder type estimates when the coefficients have partially Dini mean oscillation. Moreover, as an application, we achieve $k^{\text{th}}$ higher-order boundary Harnack principles for uniformly parabolic equations with Hölder coefficients, extending a recent result in \cite{AudFioVit24b} from $k\ge2$ to any $k\ge1$.

\end{abstract}

\maketitle

\section{Introduction}

\subsection{Degenerate or singular parabolic systems}\label{subsec:system}

Let $\al\in(-1,\infty)$ and $Q_1^+=(-1,0)\times B_1^+\times(0,1)\subset\R^{n+1}$, $n\ge2$. We consider a parabolic system in divergence form with a weight $x_n^\al$
\begin{align}\label{eq:pde-par}
    \begin{cases}
        D_\g(x_n^\al A^{\g\delta}D_\delta u)-x_n^\al\partial_tu=\div(x_n^\al\bg)\\
        \lim_{x_n\to0+}x_n^\al(A^{n\delta}D_\delta u-g^n)=0
    \end{cases}
    \quad\text{in }Q_1^+,
\end{align}
where $u=(u^1,\ldots,u^m)^{\tr}$, $m\ge1$, is a (column) vector-valued function. The coefficients $A^{\gamma\delta}=(a^{\gamma\delta}_{ij})^m_{i,j=1}$ are $m\times m$-matrices, $1\le \g,\de\le n$, and satisfies the following conditions for some constant $\la>0$
\begin{align}
\label{eq:assump-coeffi}
\begin{cases}
    \text{the uniform parabolicity: }\,\,\,\lambda|\xi|^2\le a^{\gamma\delta}_{ij}(X)\xi_i^\gamma\xi_j^\delta,\quad\xi=(\xi_i^\g)\in\R^{m\times n},\,\,\,X\in Q_1^+, \\
    \text{the uniform boundedness: }\,\,\,|A^{\g\de}(X)|\le 1/{\lambda},\quad X\in Q_1^+.
\end{cases}
\end{align}
Here $\bg=(g^1,\ldots, g^n)$, where $g^\g=(g^\gamma_1,\ldots,g^\gamma_m)^{\tr}$, $1\le\g\le n$. Throughout this paper, the summation convention over repeated indices is used. We sometimes denote $A^{\g\de}$ by $A$ for abbreviation.

We say that $u\in H^1(Q_1^+,x_n^\al dX)$ is a (weak) solution of \eqref{eq:pde-par} if
$$
\int_{Q_1^+}x_n^\al(\mean{A^{\g\de}D_\de u, D_\g\phi}-u\partial_t\phi)dX=\int_{Q_1^+}x_n^\al g^\g \cdot D_\g\phi\,dX
$$
holds for every test function $\phi\in C^\infty_c(Q_1;\R^m)$. When $\al\ge1$, it is sufficient to test the system with $\phi\in C^\infty_c(Q_1^+;\R^m)$. This can be seen by using a suitable cutoff function and the Cauchy-Schwarz inequality.

Equations with singular or degenerate coefficients appear widely in both pure and applied mathematics, such as in equations involving fractional Laplacian or fractional heat operator, porous medium equations, and problems in mathematical finance; see e.g., \cites{CafSil07, StiTor17, DasHam98, DasHamLee01, Hes93, FeeCam14}, among many others.

In this paper, we are interested in studying the degenerate-singular system \eqref{eq:pde-par} when coefficients and data satisfy partially Dini mean oscillation (partially DMO) conditions. Elliptic and parabolic systems with partially DMO coefficients arise from the linear elastic laminates and composite materials.

\begin{definition}\label{def:partial-DMO}
For $\theta\in(-1,\infty)$, let $d\mu_\theta:=x_n^\theta dX$. We say that a measurable function $f$ in $Q_2^+$ is of \emph{$L^1(d\mu_\theta)$-partially Dini mean oscillation with respect to $X'$}, denoted as $L^1(d\mu_\theta)\text{-DMO}_{X'}$, in $Q_1^+$ if the function $\eta_f^\theta:(0,1)\to[0,\infty)$ defined by
    \begin{align}\label{eq:partial-DMO}
    \eta_f^\theta(r):=\sup_{X_0\in Q_1^+}\dashint_{Q_r^+(X_0)}\left|f(X)-\dashint_{Q_r'(X_0')}f(Y',x_n)dY'\right|d\mu_\theta(X)
    \end{align}
    is a Dini function, i.e., 
    $$
    \int_0^1\frac{\eta_f^\theta(r)}rdr<\infty.
    $$
\end{definition}

When $\al=0$, interior gradient estimates for solutions of \eqref{eq:pde-par} with partially DMO coefficients can be found in \cite{DonXu21}. More recently, Schauder type estimates for degenerate-singular equations with DMO coefficients in the elliptic setting were established in \cite{DonJeoVit23}. Therefore, our paper can be seen as a generalization of \cite{DonXu21} from ``uniformly parabolic'' to ``degenerate-singular'' systems as well as an extension of \cite{DonJeoVit23} from the ``elliptic'' to the ``parabolic'' framework with a more general type of coefficients. The arguments in this paper also apply to the elliptic setting and scalar equations.

\subsection{The boundary Harnack Principle}
In \cite{TerTorVit22}, the authors showed that Schauder type estimate for degenerate equations has an important application to higher-order boundary Harnack principle. Roughly speaking, it concerns higher regularity of the ratio of two solutions, up to a portion of the boundary where they both vanish. The $C^{k,\be}$ regularity for the quotient of two harmonic functions was first proved in \cite{DeSSav15}, and this result was later extended to caloric functions in \cite{BanGar16}. Recently, \cite{TerTorVit22} established the similar result for elliptic equations in divergence form with Hölder type coefficients and a right-hand side by using the aforementioned idea. The corresponding parabolic result can be found in \cites{AudFioVit24a, AudFioVit24b}. We also refer to \cite{Kuk22} for non-divergence equation in the parabolic setting, and \cites{JeoVit24, DonJeoVit23} for equations with Dini and DMO type coefficients (see \cite{DonJeoVit23}*{Definition~1.1}).

In this paper, inspired by the original idea in \cite{TerTorVit22}, we also investigate the higher-order boundary Harnack principle for the parabolic equation in divergence form with Hölder-type coefficients.

We note that since the $u^i$'s in the system \eqref{eq:pde-par} are coupled, boundary Harnack principles cannot be expected for vector-valued solutions. In fact, even Harnack principles do not generally hold for such systems.


\subsection{Main results}

We now state our first main result. Here and henceforth we use notation from Section~\ref{subsec:notation}.

\begin{theorem}[Schauder type estimates]
    \label{thm:DMO-HO-reg-par}
    Let $k\in \mathbb{N}$ and suppose that $u$ is a solution of \eqref{eq:pde-par}.
    If $A$ satisfies \eqref{eq:assump-coeffi} and $A,\bg\in \cH^{k-1}_{\al^-}(\overline{Q_1^+})$, then $u\in \mathring{C}^{k/2,k}(\overline{Q^+_{1/2}})$.
\end{theorem}

We impose $A,\bg\in\cH^{k-1}_{\al^-}$ instead of the more natural condition $\cH_\al^{k-1}$ due to some technical difficulties in establishing the regularity of the solution in time; see the proof of Lemma~\ref{lem:mean-diff-est}. However, when the solution is time-independent, our argument remains valid under the weaker condition $A,\bg\in \cH_{\al}^{k-1}$; see also Theorem~\ref{thm:DMO-reg}. Thus, Theorem~\ref{thm:DMO-HO-reg-par} generalizes its elliptic counterpart in \cite{DonJeoVit23}.

\begin{remark}
    \label{rem:DMO-HO-reg-par}
    In Theorem~\ref{thm:DMO-HO-reg-par}, $D_x^i\partial_t^ju$, $i+2j=k$, are continuous in $X$ except $D_n^ku$.
\end{remark}

In fact, $D_n^ku$ is generally discontinuous in $x_n$. A simple counterexample is $u_0(x_n)$ in \eqref{eq:u_0-def}, which is a solution to equation \eqref{eq:u_0-pde}.

Recently, a similar result was obtained for scalar equations in \cites{AudFioVit24a, AudFioVit24b}, where the coefficients are in the parabolic Hölder space. Their approach relies on the blowup argument and a Liouville theorem, which seem to be specialized for scalar equations and the homogeneous power-type moduli. In contrast, we establish Theorem~\ref{thm:DMO-HO-reg-par} through a novel   application of Campanato's method and finite difference approximation.

For some parts of the proof, we follow the lines in \cite{DonJeoVit23}, which achieved a similar result in the elliptic context with DMO coefficients. However, the main difficulty arising in our parabolic framework is to show that when $A,g\in \cH$, a difference quotient $\frac{u(t+h,x)-u(t,x)}{h^{1/2}}$ has a Dini decay; see Theorem~\ref{thm:sol-time-est}. This step is crucial for extending Theorem~\ref{thm:DMO-HO-reg-par} from $k=1$ to $k=2$. To prove the decay estimate, we adopt Campanato’s method, inspired by the approach used in \cite{DonXu21} for uniformly parabolic equations. However, in our case, the presence of degenerate or singular terms significantly increases the complexity of the proof.

One can infer from the proof of Theorem~\ref{thm:DMO-HO-reg-par} that derivatives of the solution $u\in \mathring{C}^{k/2,k}$ has a modulus of continuity comparable with
\begin{align}\label{eq:mod-conti-sigma}\begin{split}
    \sigma_k(r):=&\left(\|u\|_{L^1(Q_4^+,d\mu)}+\|\bg\|_{L^\infty(Q_4^+)}+\int_0^1\frac{\eta_{\bg,k}^{\al^-}(\rho)}\rho d\rho\right)\times \\
    &\quad\times\left(\int_0^r\frac{\eta_{A,k}^{\al^-}(\rho)}{\rho}d\rho+r^\be\int_r^1\frac{\eta_{A,k}^{\al^-}(\rho)}{\rho^{1+\be}}d\rho+r^\be
 \right)\\
 &\qquad+\int_0^r\frac{\eta_{\bg,k}^{\al^-}(\rho)}\rho d\rho+r^\be\int_r^1\frac{\eta_{\bg,k}^{\al^-}(\rho)}{\rho^{1+\be}}d\rho
\end{split}\end{align}
for any chosen $0<\be<1$. More precisely, if $i+2j=k$ for some $i,j\in \mathbb{Z}_+$, then
$$
|D_x^i\partial_t^ju(t,x',x_n)-D_x^i\partial_t^ju(s,y',x_n)|\lesssim\sigma_k(|x'-y'|+|t-s|^{1/2}).
$$
Moreover, if $i+2j=k-1$ for some $i,j\in \mathbb{Z}_+$ and $0<h<1/4$, then
\begin{align*}
    \left|\frac{\delta_{t,h}D_x^i\partial_t^ju}{h^{1/2}}(t,x)-\frac{\delta_{t,h}D_x^j\partial_t^ju}{h^{1/2}}(s,y)\right|\lesssim \sigma_k(|x-y|+|t-s|^{1/2}).
\end{align*}
In particular, if $A,\bg$ belong to $C^{\frac{k+\bar\be-1}2,k+\bar\be-1}$ for some $\bar\be\in(0,1)$ instead of $\cH^{k-1}$, one can follow the proof of Theorem~\ref{thm:DMO-HO-reg-par} with $\be\in(\bar\be,1)$ to obtain $u\in C^{\frac{k+\bar\be}2,k+\bar\be}$. This means our result fully generalizes the parabolic Schauder estimates from the Hölder setting to the partially DMO setting.

\medskip

Our next main result concerns the boundary Harnack principle.

\begin{theorem}[Parabolic boundary Harnack principle]
    \label{thm:par-BHP}
Let $k\in\mathbb{N}$ and $\be\in(0,1)$. Suppose that two functions $u,v:Q_1^+\to\R$ satisfy
\begin{align*}
    \begin{cases}
        \div(A\D u)-\partial_tu=f,\,\,\, u>0&\text{in }Q_1^+,\\
        \div(A\D v)-\partial_tv=g&\text{in }Q_1^+,\\
        u=v=0,\,\,\,\partial_{n} u>0&\text{on }Q_1',
    \end{cases}
\end{align*}
Assume that $A:Q_1^+\to\R^{n\times n}$ is symmetric and satisfies for some $\la>0$ that $\la|\xi|^2\le\mean{A(X)\xi,\xi}$ and $|A(X)|\le1/\la$ for all $\xi\in\R^n$ and $X\in Q_1^+$. If $A,f,g\in C^{\frac{k-1+\be}2,k-1+\be}(\overline{Q_1^+})$, then $v/u\in C^{\frac{k+\be}2,k+\be}(\overline{Q^+_{1/2}})$.
\end{theorem}

Due to some technical difficulties, we address the parabolic boundary Harnack principle only for the case with Hölder-type coefficients. We leave the case with (partially) DMO-type coefficients as an open question. Concerning Theorem~\ref{thm:par-BHP}, the result was achieved in \cite{AudFioVit24b} when $k\ge2$. The approach was to reduce the problem to the degenerate equation for $w=v/u$
$$
\div(x_n^2A\D w)-x_n^2\partial_tw=\div(x_n^2\bh)+x_n^2\mean{\bb,\D w}\quad\text{for some functions }\bh,\bb.
$$
If $k\ge2$, then $u/x_n,v/x_n\in C^{\frac{1+\be}2,1+\be}$, which implies $\D w=\D\left((v/x_n)(u/x_n)^{-1}\right)\in C^{\be/2,\be}$. This allowed the authors to consider $\mean{\bb,\D w}\in C^{\be/2,\be}$ as a forcing term and the regularity of $w$ follows from a direct application of the parabolic Schauder estimate.

In this paper, we achieve Theorem~\ref{thm:par-BHP} for any $k\ge1$ by reducing to a different degenerate equation:
\begin{align*}
    \div(x_n^2A\D w)-x_n^2\partial_tw=x_n(\tilde f+\mean{\tilde\bb,\D w})\quad\text{for some functions }\tilde f,\tilde\bb.
\end{align*}
Here, $|\tilde\bb|$ is small near $Q_1'$. When $k=1$, after some modifications, we apply the Schauder estimate to show that $C^{\be/2,\be}$-norm of $\D w$ in a smaller cylinder is bounded by a quantity involving a $C^{\be/2,\be}$-norm of $\D w$ in a larger cylinder, multiplied by a small constant. We then use an iteration argument to absorb the norm of $\D w$ on the right-hand side.


\subsection{Notation and structure of the paper}\label{subsec:notation}
Throughout the paper, we shall use $X=(t,x',x_n)=(t,x)=(X',x_n)$ to denote a point in $\R^{n+1}$, where $x'=(x_1,\ldots,x_{n-1})$, $x=(x',x_n)$ and $X'=(t,x')$. Similarly, we also write $Y=(s,y)=(s,y',y_n)=(Y',y_n)$ and $X_0=(t_0,x_0)=(t_0,x_0',(x_0)_n)=(X_0',(x_0)_n)$, etc.

For $X\in\R^{n+1}$, we denote
\begin{align*}
    &B'_r(x'):=\{y'\in\R^{n-1}\,:\, |y'-x'|<r\},\quad D_r(x):=B'_r(x')\times(x_n-r,x_n+r),\\
    &D_r^+(x):=B'_r(x')\times((x_n-r)^+,x_n+r),\quad Q'_r(X'):=(t-r^2,t)\times B'_r(x'),\\
    &Q_r(X):=(t-r^2,t)\times D_r(x),\quad Q_r^+(X):=(t-r^2,t)\times D^+_r(x),\\
&\mathbb{Q}_r(X):=(t-r^2,t+r^2)\times D_r(x).
\end{align*}
For abbreviation, we simply write $D_r=D_r(0)$, $D_r^+=D_r^+(0)$, $Q'_r=Q'_r(0)$, etc.

For $\al\in\R$, we write $\al^+:=\max\{\al,0\}$ and $\al^-:=\min\{\al,0\}$.

Let $\mathbb{N}$ be the set of natural numbers and $\mathbb{Z}_+=\mathbb{N}\cup\{0\}$ be that of nonnegative integers.

Given $f:Q_1^+\to\R^m$, $h\in(0,1)$, and $1\le\g\le n$, define
\begin{align*}
    &\de_{t,h}f(t,x):=f(t,x)-f(t-h,x),\quad\de_{x_\g,h}f(t,x):=f(t,x)-f(t,x-he_\g),\\
    &\de_{t,h}^2f(t,x):=\de_{t,h}(\delta_{t,h}f(t,x))=f(t,x)-2f(t-h,x)+f(t-2h,x),\\
    &\de_{x_\g,h}^2f(t,x):=f(t,x)-2f(t,x-he_\g)+f(t,x-2e_\g).
\end{align*}

Let $\cD\subset\R^n$ be a bounded open set in $\R^n$ and $\cQ:=(a,b)\times\cD$ for some $a,b\in\R$. We define the parabolic H\"older classes $C^{l/2,l}$, for $l=k+\be$, $k\in \mathbb{Z}_+$, $0<\be\le1$, as follows. We let
\begin{align*}
    &[u]_\cQ^{(k)}:=\sum_{|\mathbf\al|+2j=k}\|D_x^{\mathbf\al}\partial_t^ju\|_{L^\infty(\cQ)},\\
    &[u]_{x,\cQ}^{(\be)}:=\sup_{(t,x),(t,y)\in\cQ}\frac{|u(t,x)-u(t,y)|}{|x-y|^\be},\quad [u]_{t,\cQ}^{(\be)}:=\sup_{(t,x),(s,x)\in\cQ}\frac{|u(t,x)-u(s,x)|}{|t-s|^\be},\\
    &[u]_\cQ^{(l)}:=\sum_{|\mathbf\al|+2j=k}[D_x^{\mathbf\al}\partial_t^ju]_{x,\cQ}^{(\be)}+\sum_{k-1\le|\mathbf\al|+2j\le k}[D_x^{\mathbf\al}\partial_t^ju]_{t,\cQ}^{\left(\frac{l-|\mathbf{\al}|-2j}2\right)}.
\end{align*}
Then, we denote $C^{l/2,l}(\cQ)$ to be the space of functions $u$ for which the following norm is finite:
$$
\|u\|_{C^{l/2,l}(\cQ)}:=\sum_{i=0}^k[u]_\cQ^{(i)}+[u]_\cQ^{(l)}.
$$

By $C_x^\be$ and $C_t^\be$, we mean the space of Hölder continuous functions with degree $\be$ in space and time, respectively. Let $C_X$ and $C_{X'}$ be the spaces of continuous functions in $X$ and $X'$, respectively.

Recall the definition of $\eta_f^\theta$ in \eqref{eq:partial-DMO}. Given $k\in\mathbb{Z}_+$ and $\theta\in(-1,\infty)$, we say that $A\in \cH^k_\theta(\overline{\cQ})$ if
\begin{align}\label{eq:coeffi-space}
    \begin{cases}
        \text{- $A\in L^\infty(\cQ)$ when $k=0$, while $A\in C^{k/2,k}(\overline{\cQ})$ when $k\ge1$,}\\
        \text{- if $i+2j=k$ for some $i,j\in \mathbb{Z}_+$, then $D_x^i\partial_t^jA$ is of $L^1(d\mu_\theta)\text{-DMO}_{X'}$ in $\cQ$,}\\
        \text{- if $i+2j=k-1$ for some $i,j\in \mathbb{Z}_+$, then for any $0<h<1/4$}\\
        \quad\text{$\frac{\delta_{t,h}D_x^i\partial_t^jA}{h^{1/2}}$ is of $L^1(d\mu_\theta)\text{-DMO}_{X'}$,}\\
        \text{- There is a Dini function $\eta_{A,k}^\theta$ such that $\eta_f^\theta\le \eta_{A,k}^\theta$ for every function}\\
        \qquad \text{$f=D_x^i\partial_t^jA,\,\, \frac{\delta_{t,h}D_x^i\partial_t^jA}{h^{1/2}}$ as above.}
    \end{cases}
\end{align}
Next, for given $k\in\mathbb{Z}_+$, we say that $u\in\mathring{C}^{k/2,k}(\overline{\cQ})$ if the following holds:
\begin{align*}
    \begin{cases}
       \text{- $u\in L^\infty(\cQ)$ when $k=0$, while $u\in C^{k/2,k}(\overline{\cQ})$ when $k\ge1$,}\\
        \text{- if $i+2j=k$ for some $i,j\in \mathbb{Z}_+$, then $D_x^i\partial_t^ju\in C_{X'}(\cQ)$,}\\
        \text{- if $i+2j=k-1$ for some $i,j\in \mathbb{Z}_+$, then $\frac{\delta_{t,h}D_x^i\partial_t^ju}{h^{1/2}}\in C_X(\cQ)$ uniformly in}\\
        \qquad\text{$h\in (0,1/4)$, and $\frac{\delta_{t,h}D_x^i\partial_t^ju}{h^{1/2}}(t,x)\to0$ as $h\to0$ uniformly in $(t,x)\in\cQ$.}
    \end{cases}
\end{align*}
We emphasize that in the second condition of $u\in \mathring{C}^{k/2,k}(\overline{\cQ})$, the derivatives $D_x^i\partial_t^ju$ may not be continuous in $x_n$. When $k=0$, we simply write $\cH^0_\theta=\cH_\theta$ and $\mathring{C}^{0,0}=\mathring{C}$.

The relation $A\lesssim B$ is understood that $A\le CB$ for some constant $C>0$.

\medskip
The rest of this paper is organized as follows. In Section~\ref{sec:prel:lem}, we provide some preliminary lemmas. In Section~\ref{sec:schauder}, we prove the Schauder-type estimate, Theorem~\ref{thm:DMO-HO-reg-par}, when $k=1$. We extend this result to any $k\in\mathbb{N}$ in Section~\ref{sec:HO-schauder}. Section~\ref{sec:BHP} is devoted to proving the Boundary Harnack principle.


\section{Preliminary lemmas}\label{sec:prel:lem}

The following are some properties of $L^1(d\mu)\text{-DMO}_{X'}$ functions.

\begin{lemma}\label{lem:product-par-DMO-elliptic}
     For $\al>-1$, if $f$ and $g$ are bounded and of $L^1(x_n^\al dX)$-$\DMO_{X'}$ in $Q_1^+$, then $fg$ is of $L^1(x_n^\al dX)$-$\DMO_{X'}$ in $Q_{1/2}^+$.
\end{lemma}

\begin{lemma}
    \label{lem:partial-DMO-weight-rela}
    Let $\al\ge\be>-1$. If $f$ is of $L^1(x_n^\be dX)$-$\DMO_{X'}$ in $Q_1^+$, then it is of $L^1(x_n^\al dX)$-$\DMO_{X'}$ in $Q^+_{1/2}$. Moreover,
    $$
    \eta_f^\al(r)\le C(n,\al,\be)\eta_f^\be(r).
    $$
\end{lemma}

The proofs of Lemma~\ref{lem:product-par-DMO-elliptic} - Lemma~\ref{lem:partial-DMO-weight-rela} follow the same line of arguments with straightforward modifications as those of their elliptic and DMO counterparts in \cite{DonJeoVit23}*{Lemma~A.2 - Lemma~A.3}. Thus, we omit the proofs.

Next, we establish the weak type-(1,1) estimate, Lemma~\ref{lem:weak-type-(1,1)}, by following lines in \cite{DonXu21}. It will be a crucial ingredient in the proof of $\mathring{C}^{1/2,1}$ regularity of solutions.

\begin{lemma}
    \label{lem:weak-type-(1,1)-aux}
    Let $X_0\in Q_1^+$, $r\in(0,1)$, and $d\mu=x_n^\al dX$ for some $\al>-1$. For a smooth convex domain $\cD_r(x_0)\subset \R^n_+$ with $D^+_{r}(x_0)\subset \cD_r(x_0)\subset D^+_{\frac43r}(x_0)$, we let $\mathcal{Q}_r(X_0):=(t_0-r^2,t_0)\times\cD_r(x_0)$. Let $T$ be a bounded linear operator from $L^2(\cQ_r(X_0);\R^{m\times n},d\mu)$ to $L^2(\cQ_r(X_0);\R^{m\times n},d\mu)$. Suppose there exists a constant $C_0>0$ such that for any $Y_0\in\cQ_r(X_0)$ and $0<\rho<r/2$,
    $$
    \int_{\cQ_r(X_0)\setminus \bbQ_{2\rho}(Y_0)}|T\bb|d\mu\le C_0\int_{Q_\rho(Y_0)\cap \cQ_r(X_0)}|\bb|d\mu
    $$
    whenever $\bb\in L^2(\cQ_r(X_0);\R^{m\times n},d\mu)$ is supported in $Q_\rho(Y_0)\cap \cQ_r(X_0)$ and satisfies $\int_{\cQ_r(X_0)}\bb d\mu=0$. Then for every $\bff\in L^2(\cQ_r(X_0);\R^{m\times n},d\mu)$ and $\tau>0$, we have
    $$
    \mu(\{X\in\cQ_r(X_0)\,:\, |T\bff(X)|>\tau\})\le\frac{C}{\tau}\int_{\cQ_r(X_0)}|\bff|d\mu,
    $$
    where $C>0$ is a constant depending only on $n,m,\al,\la$, and $C_0$.
\end{lemma}

\begin{proof}
    When $\al=0$ (i.e., $d\mu$ is the Lebesgue measure), we refer to \cite{Ste93}, where the author employs the Calderón–Zygmund decomposition and the domain considered is the whole space. In our context, we can modify the proof by incorporating the “dyadic parabolic cube" decomposition; see \cite{Chr90}*{Theorem~1.1}. This argument is applicable to our case involving a weighted measure, as outlined in \cite{DonJeoVit23}*{Lemma~2.3}.
\end{proof}

\begin{lemma}\label{lem:weak-type-(1,1)}
    Let $X_0$, $r$, and $d\mu$ be as in Lemma~\ref{lem:weak-type-(1,1)-aux}. Let $\cD_r(x_0)$ and $\tilde\cD_r(x_0)$ be smooth and convex domains in $\R^n_+$ with $D^+_{r}(x_0)\subset\cD_r(x_0)\subset D^+_{\frac43r}(x_0)$ and $D^+_{\frac32r}(x_0)\subset\tilde\cD_r(x_0)\subset D^+_{2r}(x_0)$, and set $\cQ_r(X_0):=(t_0-r^2,t_0)\times\cD_r(x_0)$ and $\tilde\cQ_r(X_0):=(t_0-4r^2,t_0)\times\tilde\cD_r(x_0)$. For one-variable functions $\bar A^{\g\de}=[\bar a^{\g\de}_{ij}(x_n)]_{i,j=1}^m$ satisfying \eqref{eq:assump-coeffi} in $\tilde\cQ_r(X_0)$ and $\bff=(f^1,\ldots,f^n)\in H^{1,\al}(\tilde\cQ_r(X_0);\R^{m\times n})$, let $u\in H^{1,\al}(\tilde\cQ_r(X_0);\R^m)$ be a weak solution of 
    \begin{align}
        \label{eq:pde-x_n}
        \begin{cases}
            D_\g(x_n^\al\bar A^{\g\de}D_\de u)-x_n^\al\partial_tu=\div(x_n^\al\bff\chi_{\cQ_r(X_0)})&\text{in }\tilde\cQ_r(X_0),\\
            u=0&\text{on }\partial_p\tilde\cQ_r(X_0)\setminus\{x_n=0\},\\
            \lim_{x_n\to0+}x_n^\al(\bar A^{n\de}D_\de u-f^n\chi_{\cQ_r(X_0)})=0&\text{on }\partial_p\tilde\cQ_r(X_0)\cap\{x_n=0\}.
        \end{cases}
    \end{align}
    Then there exists a constant $C=C(n,m,\la,\al)>0$ such that for any $\tau>0$
    \begin{align}
    &\mu(\{X\in\cQ_r(X_0)\,:\,|\D u(X)|>\tau\})\le \frac{C}\tau\int_{\cQ_r(X_0)}|\bff|d\mu,\label{eq:weak-type-1}\\
    &\mu(\{X\in\cQ_r(X_0)\,:\,|u(X)|>\tau\})\le \frac{Cr}\tau\int_{\cQ_r(X_0)}|\bff|d\mu.\label{eq:weak-type-2}
    \end{align}
\end{lemma}

\begin{proof}
For simplicity, we set $t_0=0$. We first prove \eqref{eq:weak-type-1}. Given $\hat\bff=(\hat f^1,\ldots,\hat f^n)\in H^{1,\al}(\tilde\cQ_r(X_0);\R^{m\times n})$, we consider $\bff=(f^1,\ldots, f^n)$ defined by
\begin{align}\label{eq:invert-map}
    \begin{cases}
        f_j^\g:=\hat f^\g_j+\bar a^{n\g}_{ji}\hat f^n_i&\text{when }1\le\g\le n-1,\,1\le j\le m,\\
        f_j^n:=\bar a^{nn}_{ji}\hat f^n_i&\text{when }1\le j\le m.
    \end{cases}
\end{align}
Here, $\bff$ is defined so that for any $v$,
\begin{align}
    \label{eq:f-hat-f}
    \mean{\D v,\bff}=\mean{(\D_{x'}v,V),\hat\bff}, \quad\text{where }V=(\bar A^{n\de})^{\tr}D_\de v.
\end{align}
Now we let $T(\hat\bff)=\D u$, where $u$ is a solution of \eqref{eq:pde-x_n} with $\bff$ as in \eqref{eq:invert-map}. This map $T$ is well-defined and is a bounded linear operator on $L^2(\cQ_r(X_0);\R^{m\times n},d\mu)$ by \cite{DonPha23}*{Lemma~3.4}.

To show that the operator $T$ satisfies the hypothesis of Lemma~\ref{lem:weak-type-(1,1)-aux}, let $Y_0\in \cQ_r(X_0)$ and $0<\rho<r/2$, and suppose that $\hat\bff\in L^2(\cQ_r(X_0);\R^{m\times n},d\mu)$ is a function supported in $Q_\rho(Y_0)\cap \cQ_r(X_0)$ and $\int_{Q_\rho(Y_0)\cap \cQ_r(X_0)}\hat\bff d\mu=0$. As above, let $\bff$ be as in \eqref{eq:invert-map}. For any $R\ge 2\rho$ with $\cQ_r(X_0)\setminus \bbQ_R(Y_0)\neq\emptyset$ and a function $\bh=(h^1,\ldots,h^n)\in C_c^{\infty}((\bbQ_{2R}(Y_0)\setminus \bbQ_R(Y_0))\cap\cQ_r(X_0);\R^{m\times n})$, let $v\in H^{1,\al}((-4r^2,0)\times\tilde \cD_r(x_0);\R^m)$ be a solution of 
\begin{align}
    \label{eq:weak-sol-v}
    \begin{cases}
        D_\g(x_n^\al(\bar A^{\de\g} )^{\text{tr}}D_\de v)+x_n^\al\partial_tv=\div(x_n^\al\bh)\qquad\text{in }(-4r^2,0)\times\tilde \cD_r(x_0),\\
        v=0\qquad\text{on }(\{0\}\times\tilde \cD_r(x_0))\cup((-4r^2,0)\times\partial\tilde \cD_r(x_0))\setminus \{x_n=0\})
        ,\\
        \lim_{x_n\to0+}x_n^\al((\bar A^{\de n})^{\tr}D_\de v-h^n)=0\qquad\text{on }((-4r^2,0)\times\partial\tilde \cD_r(x_0))\cap \{x_n=0\}.
    \end{cases}
\end{align}
By using \eqref{eq:pde-x_n}, \eqref{eq:f-hat-f}, \eqref{eq:weak-sol-v} and \cite{DonPha23}*{Proposition~4.4}, we can follow the argument in the proof of \cite{DonXu21}*{Lemma~4.3} to show that $T$ satisfies the hypothesis of Lemma~\ref{lem:weak-type-(1,1)-aux}, which readily implies \eqref{eq:weak-type-1}. We omit the details.

The second weak type-(1,1) estimate \eqref{eq:weak-type-2} can be derived in a similar way as in the first one \eqref{eq:weak-type-1}; see \cite{DonXu21}*{Lemma~4.3 and Lemma~5.2}. The main difference is that instead of $T:\hat\bff\longmapsto\D u$ and a solution $v$ of \eqref{eq:weak-sol-v}, we consider an operator $S:\hat\bff\longmapsto u$ and a solution $w$ of the following type of equation
\begin{align*}
    \begin{cases}
        D_\g(x_n^\al(\bar A^{\de\g})^{\tr}D_\de w)+x_n^\al\partial_tw=x_n^\al h\qquad\text{in }(-4r^2,0)\times\tilde \cD_r(x_0),\\
        w=0\qquad\text{on }(\{0\}\times\tilde \cD_r(x_0))\cup((-4r^2,0)\times\partial\tilde \cD_r(x_0))\setminus \{x_n=0\})
        ,\\
        \lim_{x_n\to0+}x_n^\al(\bar A^{\de n})^{\tr} D_\de w=0\qquad\text{on }((-4r^2,0)\times\partial\tilde \cD_r(x_0))\cap \{x_n=0\},
    \end{cases}
\end{align*}
where $h\in C_c^\infty((\bbQ_{2R}(Y_0)\setminus \bbQ_{R}(Y_0))\cap \cQ_r(X_0);\R^m)$.
\end{proof}


\section{Schauder type Estimates}\label{sec:schauder}
In this section, we prove Theorem~\ref{thm:DMO-HO-reg-par} when $k=1$.

\subsection{The continuity of \texorpdfstring{$\D_{x'}u$}{} and \texorpdfstring{$U$}{}}\label{subsec:reg-space}
The objective of this section is to derive the following result by using the ideas in \cites{Don12, DonJeoVit23}.

\begin{theorem}\label{thm:DMO-reg}
    For $\al>-1$, let $u$ be a solution of \eqref{eq:pde-par} in $Q_4^+$. Suppose the coefficients $A$ satisfy \eqref{eq:assump-coeffi} and $A,\bg\in \cH_\al(Q_4^+)$. Then $u$ is Lipschitz in $\overline{Q_1^+}$  with respect to $x$-variable. Moreover, $\D_{x'}u$ and $A^{n\de}D_\de u-g^n$ are continuous in $\overline{Q_1^+}$.
\end{theorem}

We remark that we imposed the weaker condition $A,\bg\in \cH_\al$ instead of $\cH_{\al^-}$.

We will establish an a priori estimate for the modulus of continuity of $$(\D_{x'}u,A^{n\de}D_\de u-g^n)$$ under the assumption $u$ is Lipschitz in $x$ in $\overline{Q_3^+}$. The general case can be obtained by a standard approximation argument.

In Section~\ref{subsec:reg-space}, we fix $\al>-1$ and write for simplicity
\begin{align}\label{eq:not}
d\mu=x_n^\al dX,\quad \eta_A=\eta_A^\al,\quad \eta_\bg=\eta_\bg^{\al}.
\end{align}

Given $X_0=(X_0',(x_0)_n)\in \overline{Q_3^+}$, we put $d_{X_0}:=\dist(X_0,Q'_3)=(x_0)_n\ge0$. We set
\begin{align*}
    U:=A^{n\de}D_\de u-g^n,\quad\text{and}\quad U^{X_0}:=(x_n/d_{X_0})^\al(A^{n\de}D_\de u-g^n)\,\text{ when }d_{X_0}>0.
\end{align*}
We fix $0<p<1$ and consider
\begin{align*}
    \psi(X_0,r):=\begin{cases}
        \inf_{\bq\in\R^{m\times n}}\left(\dashint_{Q_r(X_0)}|(\D_{x'}u,U^{X_0})-\bq|^pd\mu\right)^{1/p},& 0<r\le d_{X_0}/4,\\
        \inf_{\bq'\in\R^{m\times(n-1)}}\left(\dashint_{Q_r^+(X_0)}|(\D_{x'}u-\bq',U)|^pd\mu\right)^{1/p},& d_{X_0}/4< r<1/2.
    \end{cases}
\end{align*}

Next, we consider a few of Dini functions stemmed from $\eta_\bullet$, where $\bullet$ is either $A$ or $\bg$. For some constants $C>c>0$, depending only on $n$ and $\al$, we have
\begin{align}
    \label{eq:alm-mon}
    c\eta_\bullet(r)\le \eta_\bullet(s)\le C\eta_\bullet(r)
\end{align}
whenever $r/2\le s\le r<1$. See e.g. \cite{Li17}. Given constant $\be\in(0,1)$, we write $\be'=\frac{\be+1}2$ so that $\be<\be'<1$. For a small constant $0<\ka<1$, we set
\begin{align}\label{eq:tilde-eta}
\tilde\eta_\bullet(r):=\sum_{i=0}^\infty\ka^{\be' i}\eta_\bullet(\ka^{-i}r)[\ka^{-i}r\le 1],\quad 0<r<1,
\end{align}
where we used Iverson's bracket notation (i.e., $[P]=1$ when $P$ is true, while $[P]=0$ otherwise). We also let
\begin{align}
    \label{eq:hat-eta}
    \hat\eta_\bullet(r):=\sup_{\rho\in [r,1)}(r/\rho)^{\beta'}\tilde\eta_\bullet(\rho),\quad 0<r<1.
\end{align}
Then $\tilde\eta_\bullet$ and $\hat\eta_\bullet$ are Dini functions satisfying \eqref{eq:alm-mon} and $\hat\eta_\bullet\ge\tilde\eta_\bullet\ge \eta_\bullet$. Moreover, $r\longmapsto \frac{\hat\eta_\bullet(r)}{r^{\beta'}}$ is nonincreasing. See e.g., \cite{Don12}.

Now we begin the proof of Theorem~\ref{thm:DMO-reg},  primarily following the approach in \cite{DonJeoVit23}, which establishes a similar result in the elliptic setting with DMO coefficients. We address additional difficulties arising in the partial DMO setting by using the idea in \cite{Don12}.

\begin{lemma}\label{lem:psi-bdry-est}
    Let $0<p<1$, $0<\be<1$, $\be'=\frac{1+\be}2$ and $\bar X_0\in Q_3'$. Then for any $0<\rho<r<1/2$, we have
    \begin{align}\label{eq:psi-bdry-est}
        \psi(\bar X_0,\rho)\le C(\rho/r)^{\beta'}\psi(\bar X_0,r)+C\|\D u\|_{L^\infty(Q_{2r}^+(\bar X_0))}\tilde\eta_A(2\rho)+C\tilde\eta_\bg(2\rho),
    \end{align}
    where $C=C(n,m,\la,\al,p,\be)>0$ and $\tilde\eta_\bullet$ is as in \eqref{eq:tilde-eta}.
\end{lemma}

\begin{proof}
We assume without loss of generality $\bar X_0=0$. For a fixed $0<r<1/2$, we set
\begin{align*}
    &\bar A^{\g\de}(x_n):=\dashint_{Q_{2r}'}A^{\g\de}(Y',x_n)dY',\\
    &\bar\bg(x_n)=(\bar g^1(x_n),\dots,\bar g^n(x_n)):=\dashint_{Q'_{2r}}\bg(Y',x_n)dY'.
\end{align*}
Note that $u$ satisfies
$$
D_\g(x_n^\al\bar A^{\g\de}D_\de u)-x_n^\al\partial_tu=D_\g(x_n^\al((\bar A^{\g\de}-A^{\g\de})D_\de u+g^\g)).
$$
We consider a function 
\begin{align}\label{eq:u_0-def}
u_0(x_n):=\int_0^{x_n}(\bar A^{nn}(y_n))^{-1}\bar g^n(y_n)dy_n\quad \text{in }Q_{2r}^+,
\end{align}
which solves
\begin{align}\label{eq:u_0-pde}
\begin{cases}
   D_\g(x_n^\al\bar A^{\g\de}D_\de u_0)-x_n^\al\partial_tu_0=\div(x_n^\al\bar\bg)\\
    \lim_{x_n\to0+}x_n^\al(\bar A^{n\de}D_\de u_0-\bar g^n)=0
\end{cases} \text{in } Q^+_{2r}.
\end{align}
Then, $u_e(X',x_n):=u(X',x_n)-u_0(x_n)$ satisfies
$$
\begin{cases}
D_\g(x_n^\al\bar A^{\g\de}D_\de u_e)-x_n^\al\partial_tu_e=D_\g(x_n^\al((\bar A^{\g\de}-A^{\g\de})D_\de u+g^\g-\bar g^\g))\\
\lim_{x_n\to0+}x_n^\al(\bar A^{n\de}D_\de u_e+(A^{n\de}-\bar A^{n\de})D_\de u+\bar g^n-g^n)=0
\end{cases}\text{in }Q_{2r}^+.
$$
We take smooth and convex domains $\cD_r$ and $\tilde \cD_r$ such that $D_r^+\subset \cD_r\subset D_{\frac43r}^+$ and $D^+_{\frac32r}\subset \tilde \cD_r\subset D^+_{2r}$, and write $\cQ_r:=\cD_r\times(-r^2,0)$ and $\tilde\cQ_r:=\tilde D_r\times(-4r^2,0)$. Let $w$ be a solution of
\begin{align*}
\left\{
    \begin{aligned}
        D_\g(x_n^\al\bar A^{\g\de}D_\de w)-x_n^\al\partial_tw=D_\g(x_n^\al((\bar A^{\g\de}-A^{\g\de})D_\de u+g^\g-\bar g^\g)\chi_{\cQ_r}) \quad\text{in }\tilde \cQ_r,\\
        w=0\quad\text{on }\partial_p\tilde\cQ_r\setminus\{x_n=0\},\\
        \lim_{x_n\to0+}x_n^\al(\bar A^{n\de}D_\de w-((\bar A^{n\de}-A^{n\de})D_\de u+g^n-\bar g^n)\chi_{\cQ_r})=0\quad\text{on }\partial_p\tilde \cQ_r\cap\{x_n=0\}.
    \end{aligned}
\right.
\end{align*}
By applying the weak type-(1,1) estimate \eqref{eq:weak-type-1} and following the same argument as in obtaining \cite{DonJeoVit23}*{(2.8)}, we get
\begin{align}\label{eq:w-est}
    \left(\dashint_{Q_r^+}|\D w|^pd\mu\right)^{1/p}\le C\|\D u\|_{L^\infty(Q^+_{2r})}\eta_A(2r)+C\eta_\bg(2r).
\end{align}

Next, we let $v:=u_e-w$. It is easily seen that $v$ satisfies
\begin{align}\label{eq:hom-repla-pde}
    \begin{cases}
        D_\g(x_n^\al\bar A^{\g\de}D_\de v)-x_n^\al\partial_tv=0\\
        \lim_{x_n\to0+}x_n^\al\bar A^{n\de}D_\de v=0
    \end{cases}\text{in }Q_r^+.
\end{align}
Since $D_iv$ satisfies \eqref{eq:hom-repla-pde} for $1\le i\le n-1$, we have by \cite{DonPha23}*{Lemma~4.2 and Proposition~4.4} (see also \cite{BeDo25} for systems)
$$
\|DD_{x'}v\|_{L^\infty(Q_{r/2}^+)}\le \frac{C}r\left(\dashint_{Q_r^+}|\D_{x'}v|^2d\mu\right)^{1/2}\le \frac{C}r\left(\dashint_{Q_r^+}|\D v|^2d\mu\right)^{1/2}.
$$
Similarly, since $\partial_tu$ satisfies \eqref{eq:hom-repla-pde}, from \cite{DonPha23}*{Lemmas~4.2 and 4.3} (see also \cite{BeDo25}) we have
$$
\|\partial_tv\|_{L^\infty(Q_{r/2}^+)}\le\frac{C}r\left(\dashint_{Q_r^+}|\D v|^2d\mu\right)^{1/2}.
$$
By using a standard iteration as in \cite{Gia93}*{pages 80-82}, we improve the previous two estimates:
\begin{align}
    \label{eq:hom-repla-deriv-est}
    \|DD_{x'}v\|_{L^\infty(Q_{r/2}^+)}+\|\partial_tv\|_{L^\infty(Q_{r/2}^+)}\le\frac{C}r\left(\dashint_{Q_r^+}|\D v|^pd\mu\right)^{1/p}.
\end{align}
We put $V:=\bar A^{n\de}D_\de v$ in $Q_r^+$. From the first equation in \eqref{eq:hom-repla-pde}, we infer
$$
D_n(x_n^\al V)=x_n^\al\partial_tv-x_n^\al\sum_{\g=1}^{n-1}\sum_{\de=1}^n\bar A^{\g\de}D_{\g\de}v.
$$
By combining this with the conormal boundary condition in \eqref{eq:hom-repla-pde} and the estimate \eqref{eq:hom-repla-deriv-est}, we deduce that for every $X\in Q^+_{r/2}$
\begin{align*}
    |x_n^\al V(X)|&=\left|\int_0^{x_n}\partial_s(s^\al V(X',s))ds\right|\\
    &\le \int_0^{x_n}s^\al\left(|\partial_tv(X',s)|+\sum_{\g=1}^{n-1}\sum_{\de=1}^n|\bar A^{\g\de}D_{\g\de}v(X',s)|\right)ds\\
    &\le \frac{C}r\left(\dashint_{Q_r^+}|\D v|^pd\mu\right)^{1/p}\int_0^{x_n}s^\al ds\le \frac{Cx_n^{\al+1}}{r}\left(\dashint_{Q_r^+}|\D v|^pd\mu\right)^{1/p}.
\end{align*}
This gives
\begin{align}\label{eq:V-est}
    |V(X)|\le \frac{Cx_n}{r}\left(\dashint_{Q_r^+}|\D v|^pd\mu\right)^{1/p},\quad X\in Q^+_{r/2}.
\end{align}
Moreover, by using \cite{DonPha23}*{Proposition~4.4} (see also \cite{BeDo25}) and the iteration as above, we have
$$
[\D_{x'}v]_{C^{1/2,1}(Q^+_{r/2})}\le\frac{C}r\left(\dashint_{Q_r^+}|\D v|^pd\mu\right)^{1/p}.
$$
The preceding two estimates give that for small $\ka\in(0,1/2)$ to be chosen later 
\begin{align}
    \label{eq:hom-repla}\begin{split}
    &\left(\dashint_{Q_{\ka r}^+}\left(|\D_{x'}v-\mean{\D_{x'}v}_{Q_{\ka r}^+}|^p+|V|^p\right)d\mu\right)^{1/p}\\
    &\le C\ka \left(\dashint_{Q_r^+}|\D v|^pd\mu\right)^{1/p}\le C\ka\left(\dashint_{Q_r^+}\left(|\D_{x'}v|^p+|V|^p\right)d\mu\right)^{1/p}.
\end{split}\end{align}
For any constant vector $\bq'=(q_1,\ldots q_{n-1})\in \R^{m\times(n-1)}$, we set
$$
\tilde v(X):=v(X)-\bq'\cdot x'+\int_0^{x_n}(\bar A^{nn}(y_n))^{-1}\left(\sum_{\de=1}^{n-1}\bar A^{n\de}(y_n)q_\de\right)dy_n$$
and
$$\tilde V:=\bar A^{n\de}D_\de\tilde v.
$$
It is easily seen that $\tilde v$ satisfies \eqref{eq:hom-repla-pde} and $\tilde V=V$ in $Q_r^+$. This enables us to replace $v$ and $V$ with $\tilde v$ and $\tilde V$ in \eqref{eq:hom-repla}, respectively, to have
\begin{align*}
&\left(\dashint_{Q_{\ka r}^+}\left(|\D_{x'}v-\mean{\D_{x'}v}_{Q_{\ka r}^+}|^p+|V|^p\right)d\mu\right)^{1/p}\\
&    \le C\ka\left(\dashint_{Q_r^+}\left(|\D_{x'}v-\bq'|^p+|V|^p\right)d\mu\right)^{1/p}.
\end{align*}
By recalling $u_e=v+w$, $\bar A^{n\de}D_\de u_e=V+\bar A^{n\de}D_\de w$, $u=u_e+u_0$ and $\bar A^{n\de}D_\de u_0-\bar g^n=0$ and using \eqref{eq:w-est}, we follow \cite{DonJeoVit23}*{page 13} to further have
\begin{align*}
    &\left(\dashint_{Q_{\ka r}^+}\left(|\D_{x'}u-\mean{\D_{x'}v}_{Q_{\ka r}^+}|^p+|U|^p\right)d\mu\right)^{1/p}\\
    &\le C\ka \left(\dashint_{Q_{r}^+}\left(|\D_{x'}u-\bq'|^p+|U|^p\right)d\mu\right)^{1/p}\\
    &\quad +C\ka^{-\frac{n+2+\al}p}\left(\|\D u\|_{L^\infty(Q^+_{2r})}\eta_A(2r)+\eta_\bg(2r)\right).
\end{align*}
Since $\bq'\in \R^{n-1}$ is arbitrary, we get
$$
\psi(0,\ka r)\le C\ka\psi(0,r)+C\ka^{-\frac{n+2+\al}p}\left(\|\D u\|_{L^\infty(Q^+_{2r})}\eta_A(2r)+\eta_\bg(2r)\right).
$$
By taking $\ka$ small so that $C\ka\le \ka^{\be'}$ and using a standard iteration (see \cite{DonEscKim18}*{page 461}), we obtain \eqref{eq:psi-bdry-est}.   
\end{proof}

Given $X_0\in Q_3^+$, when there is no confusion, we simply write 
$$
d=d_{X_0}.
$$

\begin{lemma}\label{lem:psi-int-est}
    Let $p$, $\be$, and $\be'$ be as in Lemma~\ref{lem:psi-bdry-est}. For any $X_0\in Q_3^+$ and $0<\rho<r<d/4$, we have
    \begin{align}\label{eq:psi-int-est}
    \psi(X_0,\rho)\le C(\rho/r)^{\be'}\psi(X_0,r)+C\|\D u\|_{L^\infty(Q_{r}(X_0))}\tilde\eta_A(\rho)+C\tilde\eta_\bg(\rho),
    \end{align}
    where $C=C(n,m,\la,\al,p,\be)>0$ are constants and $\tilde\eta_\bullet$ is as in \eqref{eq:tilde-eta}.
\end{lemma}

\begin{proof}
As $r<d/4$, $u$ solves a uniformly parabolic equation
\begin{align}
    \label{eq:unif-elliptic}
    D_\g((A^{\g\de})^{X_0}D_\de u)-a_0(x_n)\partial_tu=\div\bg^{X_0}\quad\text{in }Q_r(X_0),
\end{align}
where $(A^{\g\de})^{X_0}:=(x_n/d)^\al A^{\g\de}$, $\bg^{X_0}:=(x_n/d)^\al\bg$, and $a_0(x_n):=(x_n/d)^\al$. Notice that $c\le a_0\le c^{-1}$ in $Q_r(X_0)$ for some constant $c=c(\al)\in (0,1)$. The elliptic version of Lemma~\ref{lem:psi-int-est} is found in \cite{DonJeoVit23}*{Lemma~2.6}, which was proved using \cite{DonXu19}*{Lemma~3.3}. In our parabolic framework, we can follow the argument in \cite{DonJeoVit23}*{Lemma~2.6} by using a parabolic analogue, \cite{DonXu21}*{Lemma~4.4}. While the latter only considers the case $a_0\equiv 1$ in \eqref{eq:unif-elliptic}, it is easily seen that its proof also holds for our $a_0(x_n)=(x_n/d)^\al$ since \cite{DonPha23}*{Theorem~4.1 and Proposition~4.4} (see also \cite{BeDo25}) is applicable in this case with such $a_0$.
\end{proof}

By combining the preceding two lemmas, we can argue as in the proof of \cite{DonJeoVit23}*{Lemma~2.7} to obtain the following growth estimate of $\psi$.

\begin{lemma}
    \label{lem:psi-est}
    Let $p$, $\be$ and $\be'$ be as before. If $X_0\in Q_3^+$ and $0<\rho\le r<1/8$, then
    \begin{align}
    \label{eq:psi-est}
    \psi(X_0,\rho)\le \begin{multlined}[t]C(\rho/r)^{\be'}\left(\dashint_{Q_{5r}^+(X_0)}|\D u|d\mu+\|\bg\|_{L^\infty(Q_4^+)}\right)\\+C\left(\|\D u\|_{L^\infty(Q_{8r}^+(X_0))}\hat\eta_A(\rho)+\hat\eta_\bg(\rho)\right),
\end{multlined}\end{align}
where $C=C(n,m,\la,\al,p,\be)>0$ and $\hat\eta_\bullet$ is as in \eqref{eq:hat-eta}.
\end{lemma}

Once we have Lemma~\ref{lem:psi-est}, we can follow the approach in \cite{DonJeoVit23} to derive Theorem~\ref{thm:DMO-reg}. Therefore, we will only outline the key steps below.

For $X_0\in Q_3^+$ and $0<r<1/8$, we take a vector $\bq_{X_0,r}\in\R^n$ such that
\begin{align}\label{eq:psi-inf}
    \psi(X_0,r)=\begin{cases}
        \left(\dashint_{Q_r(X_0)}|(\D_{x'}u,U^{X_0})-\bq_{X_0,r}|^pd\mu\right)^{1/p},&0<r\le d/4,\\
        \left(\dashint_{Q_r^+(X_0)}|(\D_{x'}u, U)-\bq_{X_0,r}|^pd\mu\right)^{1/p},&d/4<r<1/8.
    \end{cases}
\end{align}
Notice that the last component of $\bq_{X_0,r}$ is zero when $d/4<r<1/8$.

\begin{lemma}\label{lem:gradient-unif-est}
It holds that
    \begin{align}\label{eq:gradient-unif-est}
        \|\D u\|_{L^\infty(Q_2^+)}\le C\int_{Q_4^+}|u|d\mu+C\int_0^1\frac{\hat\eta_\bg(\rho)}\rho d\rho+C\|\bg\|_{L^\infty(Q_4^+)},
    \end{align}
    for some constant $C>0$ depending only on $n,m,\la,\al,p$.
\end{lemma}

\begin{proof}
With Lemmas~\ref{lem:psi-bdry-est}-\ref{lem:psi-est} at hand, we can follow the argument in Step 1 in the proof of \cite{DonJeoVit23}*{Lemma~2.8} to obtain 
$$
\|\D u\|_{L^\infty(Q_2^+)}\le C\int_{Q_3^+}|\D u|d\mu+C\int_0^1\frac{\hat\eta_\bg(\rho)}\rho d\rho+C\|\bg\|_{L^\infty(Q_4^+)}.
$$
Then \eqref{eq:gradient-unif-est} follows by the Caccioppoli inequality, parabolic embedding and iteration.
\end{proof}

To proceed, given $0<\be<1$, we consider a modulus of continuity $\omega_x:[0,1)\to[0,\infty)$ defined by
\begin{align}\label{eq:omega-x}
    \begin{split}
        \omega_x(r)
        &=\left(\int_{Q_4^+}|u|d\mu+\|\bg\|_{L^\infty(Q_4^+)}+\int_0^1\frac{\hat\eta_\bg^\al(\rho)}\rho d\rho \right)\left(\int_0^r\frac{\hat\eta_A^\al(\rho)}\rho d\rho+r^{\be}\right)\\
        &\qquad+\int_0^r\frac{\hat\eta_\bg^\al(\rho)}\rho d\rho.
    \end{split}
\end{align}
Recall \eqref{eq:mod-conti-sigma}. By Lemma~\ref{lem:Dini-sum-est}, we have
\begin{align}
    \label{eq:mod-conti}\omega_x(r)&\lesssim\sigma_0(r).
\end{align}

\begin{lemma}
   Let $\be\in(0,1)$. For any $X_0\in Q_1^+$ and $0<r<1/18$, we have
    \begin{align}\label{eq:u-q-diff}\begin{split}
        &|(\D_{x'}u(X_0),U(X_0))-\bq_{X_0,r}|\le C\omega_x(r)
    \end{split}\end{align}
for some constant $C>0$ depending only on $n,m,\la,\al,p$, and $\be$.
\end{lemma}

\begin{proof}
The proof of this lemma is similar to that of \cite{DonJeoVit23}*{Lemma~2.9}.
\end{proof}

We are now ready to prove Theorem~\ref{thm:DMO-reg}.

\begin{proof}[Proof of Theorem~\ref{thm:DMO-reg}]
The Lipschitz regularity of $u$ in $x$ follows from Lemma~\ref{lem:gradient-unif-est}. Moreover, by following the argument in \cite{DonJeoVit23}*{Theorem~2.4}, we can show that for any $X_0,Y_0\in Q_1^+$ with $r:=|x_0-y_0|+|t_0-s_0|^{1/2}>0$,
\begin{equation}\label{eq:mod-conti-space-est}
    |(\D_{x'} u(X_0),U(X_0))-(\D_{x'} u(Y_0),U(Y_0))|\le C\omega_x(r).\qedhere
\end{equation}
\end{proof}

\subsection{The regularity of \texorpdfstring{$u$}{} in time}\label{subsec:reg-time}
In this section, we write for simplicity
$$
d\mu=x_n^{\al^-}dX,\quad \eta_A=\eta_A^{\al^-},\quad \eta_\bg=\eta_\bg^{\al^-}.
$$
Note that we used the same notation in Section~\ref{subsec:reg-space}, but with $\al$ in the place of $\al^-$; see \eqref{eq:not}.

We fix a constant $0<\be<1$ and recall $\be'=\frac{\be+1}2$. We introduce several Dini functions that will be used in this section. We simply write $\|\D u\|_{L^\infty(Q_2^+)}=\|\D u\|_{\infty}$, $\|u\|_{L^1(Q_4^+,d\mu)}=\|u\|_{1,\mu}$ and $\|\bg\|_{L^\infty(Q_4^+)}=\|\bg\|_{\infty}$. Set
\begin{align}
    \label{eq:Dini-tilde}\begin{split}
    &\eta(r):=(\|\D u\|_{\infty}+\|\bg\|_{\infty})\eta_A(r)+\eta_\bg(r),\\
    &\tilde\eta(r):=\left(\|u\|_{1,\mu}+\|\bg\|_{\infty}+\int_0^1\frac{\eta_\bg(\rho)}\rho d\rho\right)(\tilde\eta_A(r)+r^{\be'})+\tilde\eta_\bg(r),\\
    &\hat\eta(r):=\left(\|u\|_{1,\mu}+\|\bg\|_{\infty}+\int_0^1\frac{\eta_\bg(\rho)}\rho d\rho\right)(\hat\eta_A(r)+r^{\be'})+\hat\eta_\bg(r),\\
    &\bar\eta(r):=\sum_{k=0}^\infty\frac1{2^k}\hat\eta(r/2^k),
\end{split}\end{align}
where $\tilde\eta_\bullet$ and $\hat\eta_\bullet$ are as in \eqref{eq:tilde-eta} and \eqref{eq:hat-eta}, respectively. Note that $\eta\le\tilde\eta\le\hat\eta\le\bar\eta$. We define
\begin{align}
\label{eq:omega-t}
\omega_t(r):=\bar\eta(r^{1/2}).
\end{align}
Thanks to Lemma~\ref{lem:Dini-sum-est}, it is easily seen that for $\sigma_0$ as in \eqref{eq:mod-conti-sigma}
\begin{align}
    \label{eq:mod-conti-time}
    \omega_t(r)\lesssim\sigma_0(r^{1/2}).
\end{align}

The purpose of this section is to establish the following type of regularity for solutions in time, which will play a significant role in the higher-order Schauder type estimates in the next section.

\begin{theorem}
    \label{thm:sol-time-est}Let $u$ be a solution of \eqref{eq:pde-par} in $Q_4^+$, $A,\bg\in\cH_{\al^-}(\overline{Q_4^+})$, and $0<\be<1$. Then, for any $0<h<1/4$ and $(t,x)\in Q_1^+$, we have
    \begin{align}\label{eq:sol-time-est}
    \frac{|\delta_{t,h}u(t,x)|}{h^{1/2}}\le C\omega_t(h),
    \end{align}
    where $\omega_t$ is as in \eqref{eq:omega-t} and $C=C(n,m,\al,\la,\be)>0$.
\end{theorem}

In \cite{DonXu21}, the first author and X. Lu derived a similar result in a simpler setting, e.g., when $\al=0$ and in Hölder space, by employing the Campanato's method. They worked on $u$ itself rather than its first derivative, and introduced a set of polynomials with respect to $x'$ to estimate the difference between $u$ and some polynomials. We adopt this approach to prove Theorem~\ref{thm:sol-time-est}, but it is worth noting that our case with general $\al>-1$ requires significantly more complex techniques.

For $X_0\in Q_{3}^+$ and $0<r<1/4$, we write for $1\le \g,\de\le n$
\begin{align*}
    (\bar g^\g)^{X_0,r}(x_n)&:=\dashint_{Q'_{2r}(X'_0)}g^\g(Y',x_n)dY',\\ 
   (\bar A^{\g\de})^{X_0,r}(x_n)&:=\dashint_{Q'_{2r}(X_0')}A^{\g\de}(Y',x_n)dY'.
\end{align*}
As before, we write for simplicity 
$
d:=(x_0)_n.
$
When $r\ge d/4$, we denote
\begin{align*}
    {\mathbb{A}}_1^{X_0,r}=\bigg\{q\,:\,q(x)=&\ell_0+\sum_{\de=1}^{n-1}\ell_\de(x_\de-(x_0)_\de)\\
    &-\int_{d}^{x_n}\left((\bar A^{nn})^{X_0,r}(s)\right)^{-1}\sum_{\de=1}^{n-1}(\bar A^{n\de})^{X_0,r}(s)\ell_\de\,ds \bigg\},
\end{align*}
where $\ell_0,\ell_1,\ldots, \ell_{n-1}$ are constants in $\R^m$. On the other hand, when $0<r<d/4$, we define 
\begin{align*}
    \mathbb{A}_2^{X_0,r}=\left\{q(x)+\ell_n\int_{d}^{x_n}\left((\bar A^{nn})^{X_0,r}(s)\right)^{-1}(d/s)^\al ds\,:\, q(x)\in \mathbb{A}_1^{X_0,r} \right\},
\end{align*}
where $\ell_n$ is a constant in $\R^m$. Moreover, we set
$$
u_0^{X_0,r}(x_n):=\int_{d}^{x_n}((\bar A^{nn})^{X_0,r}(s))^{-1}(\bar g^n)^{X_0,r}(s)\,ds.
$$
We fix $0<p<1$ and consider
\begin{align*}
    \vp(X_0,r):=\begin{cases}
        \inf_{q\in \mathbb{A}_{1}^{X_0,r}}\left(\dashint_{Q_r^+(X_0)}|u-u_0^{X_0,r}-q|^pd\mu\right)^{1/p},& r\ge d/4,\\
        \inf_{q\in \mathbb{A}_{2}^{X_0,r}}\left(\dashint_{Q_r(X_0)}|u-u_0^{X_0,r}-q|^pd\mu\right)^{1/p},& 0<r< d/4.
    \end{cases}
\end{align*}
We also consider a function 
\begin{align}\label{eq:q-r}\begin{split}
    q^{X_0,r}(x)=&\ell_0^{X_0,r}+\sum_{\de=1}^{n-1}\ell_\de^{X_0,r}(x_\de-(x_0)_\de)\\
    &+\int_{d}^{x_n}((\bar A^{nn})^{X_0,r}(s))^{-1}\left((d/s)^\al\ell_n^{X_0,r}-\sum_{\de=1}^{n-1}(\bar A^{n\de})^{X_0,r}(s)\ell_\de^{X_0,r}\right)ds,\end{split}
\end{align}
which belongs to $\mathbb{A}_{1}^{X_0,r}$ when $r\ge d/4$ while $\mathbb{A}_2^{X_0,r}$ when $0<r<d/4$ and satisfies
\begin{align}\label{eq:q-phi-eqn}
\left(\dashint_{Q^+_r(X_0)}|u-u_0^{X_0,r}-q^{X_0,r}|^pd\mu\right)^{1/p}=\vp(X_0,r).
\end{align}
Here, $\ell_n^{X_0,r}=0$ when $r\ge d/4$ due to the definition of $\mathbb{A}_1^{X_0,r}$. When $d=0$, we understand $(d/s)^\al \ell_n^{X_0,r}=0$.

A crucial step in proving Theorem~\ref{thm:sol-time-est} is the estimate of $\vp$ presented in Lemma~\ref{lem:vp-est}. To obtain this estimate, we will consider the boundary case (Lemma~\ref{lem:vp-bdry-est}) and the interior case (Lemma~\ref{lem:vp-int-est}). We first establish the following technical lemma.

\begin{lemma}
    \label{lem:mean-diff-est}
    For any $X_0\in\overline{Q_3^+}$, $r\in(0,1)$ and $\ka\in(0,1)$, 
    \begin{align*}
        &\dashint_{Q_{\ka r}^+(X_0)}\int_d^{x_n}|(\bar g^n)^{X_0,\ka r}(s)-(\bar g^n)^{X_0,r}(s)|dsd\mu\lesssim \frac{r\eta_\bg^{\al^-}(r)}{\ka^{n+1}},\\
        &\dashint_{Q_{\ka r}^+(X_0)}\int_d^{x_n}|(\bar A^{\g\delta})^{X_0,\ka r}(s)-(\bar A^{\g\delta})^{X_0,r}(s)|dsd\mu\lesssim \frac{r\eta_A^{\al^-}(r)}{\ka^{n+1}}.
    \end{align*}
\end{lemma}

\begin{proof}
We only prove the first estimate since the other one can be obtained in a similar way. By applying Fubini's Theorem and Lemma~\ref{lem:partial-DMO-weight-rela}, we obtain
\begin{align*}
    &\dashint_{Q_{\ka r}^+(X_0)}\int_d^{x_n}|(\bar g^n)^{X_0,\ka r}(s)-(\bar g^n)^{X_0,r}(s)|dsd\mu\\
    &=\frac1{\int_{(d-\ka r)^+}^{d+\ka r}x_n^\al dx_n}\int_{(d-\ka r)^+}^{d+\ka r}\int_d^{x_n}|(\bar g^n)^{X_0,\ka r}(s)-(\bar g^n)^{X_0,r}(s)|x_n^\al dsdx_n\\
    &\lesssim \frac{1}{(d+\ka r)^{1+\al}-((d-\ka r)^+)^{1+\al}}\int_d^{d+\ka r}\int_s^{d+\ka r}|(\bar g^n)^{X_0,\ka r}(s)-(\bar g^n)^{X_0,r}(s)|x_n^\al dx_nds\\
    &\lesssim \int_d^{d+\ka r}|(\bar g^n)^{X_0,\ka r}(s)-(\bar g^n)^{X_0,r}(s)|ds\\
    &\lesssim\frac1{(\ka r)^{n+1}}\int_{Q^+_{2r}(X_0)}|g^n(X)-(\bar g^n)^{X_0,r}(x_n)|dX\\
    &\lesssim \frac{r\eta_\bg^0(r)}{\ka^{n+1}}\lesssim \frac{r\eta_\bg^{\al^-}(r)}{\ka^{n+1}}.\qedhere
\end{align*}
\end{proof}

\begin{lemma}
    \label{lem:vp-bdry-est}
    Let $0<p<1$, $0<\be<1$, $\be'=\frac{1+\be}2$, and $\bar X_0\in Q'_3$. Then there exists a constant $\ka\in(0,1/4)$, depending only on $n,m,\al,p,\la$ and $\be$, such that for any $0<r<1/4$,
    \begin{align}\label{eq:vp-bdry-est}
        \vp(\bar X_0,\ka r)\le \ka^{1+\be'}\vp(\bar X_0,r)+Cr\eta(r),
    \end{align}
    where $C=C(n,m,\al,p,\la)>0$.
\end{lemma}

\begin{proof}
Without loss of generality we may assume $\bar X_0=0$. We simply write $(\bar g^n)^{0,r}=(\bar g^n)^{r}$, $\mathbb{A}_1^{0,r}=\mathbb{A}_1^r$ and $u_0^{0,r}=u_0^{r}$, etc, and decompose $u=u_0^{r}+v+w$, where $v$ and $w$ are as in the proof of Lemma~\ref{lem:psi-bdry-est}. By applying the weak type-(1,1) estimate \eqref{eq:weak-type-2} and Lemma~\ref{lem:partial-DMO-weight-rela}, we get
\begin{align}\label{eq:w-est-2}
    \left(\dashint_{Q_r^+}|w|^p d\mu\right)^{1/p}\le Cr\left(\|\D u\|_{L^\infty(Q_2^+)}\eta_A(2r)+\eta_\bg(2r)\right).
\end{align}
Regarding $v$, recall that it satisfies \eqref{eq:hom-repla-pde} and $V=(\bar A^{n\de})^rD_\de v$. We define the operator $\hat T_{r}$ by
\begin{align}\label{eq:Tv}
    \hat T_{r}v(x):=v(0)+\D_{x'}v(0)\cdot x'-\int_0^{x_n}((\bar A^{nn})^{r}(s))^{-1}\sum_{\de=1}^{n-1}(\bar A^{n\de})^{r}(s)D_\de v(0)\,ds.
\end{align}
Writing $\hat q_{r}:=\hat T_{r}v$, we have for a small constant $\ka\in(0,1/4)$ to be chosen later
$$
\|\D_{x'}(v-\hat q_{r})\|_{L^\infty(Q^+_{\ka r})}=\|\D_{x'}v-\D_{x'}v(0)\|_{L^\infty(Q^+_{\ka r})}\le \ka r[\D_{x'}v]_{C^{1/2,1}(Q_{\ka r}^+)}
$$
and 
$$
\left\|\sum_{\de=1}^n(\bar A^{n\de})^{r}D_\de(v-\hat q_{r})\right\|_{L^\infty(Q^+_{\ka r})}=\|V\|_{L^\infty(Q^+_{\ka r})}\le \ka r[V]_{C^{1/2,1}(Q^+_{\ka r})},
$$
where in the last step we used $V=0$ on $Q'_{r/2}$, which can be inferred from \eqref{eq:V-est}. By combining these estimates and \cite{DonPha23}*{Lemma~4.2 and Proposition~4.4} (see also \cite{BeDo25}) with standard iteration, we get
\begin{align*}
    \|\D(v-\hat q_{r})\|_{L^\infty(Q^+_{\ka r})}&\le C\ka\left(\dashint_{Q_{r/2}^+}|\D v|^pd\mu\right)^{1/p} \le \frac{C\ka}r\left(\dashint_{Q_r^+}|v|^pd\mu\right)^{1/p}.
\end{align*}
In addition, since $\partial_tv$ satisfies \eqref{eq:hom-repla-pde}, we use \cite{DonPha23}*{Lemmas 4.2 and 4.3} (see also \cite{BeDo25}) 
to obtain
$$
\|\partial_tv\|_{L^\infty(Q^+_{r/4})}\le C\left(\dashint_{Q_{r/2}^+}|\partial_tv|^pd\mu\right)^{1/p}\le\frac{C}{r^2}\left(\dashint_{Q_r^+}|v|^pd\mu\right)^{1/p}.
$$
Combining the previous two estimates and the equality $(v-\hat q_{r})(0)=0$ yields
\begin{align}
    \label{eq:hom-repla-poly-est}
    \left(\dashint_{Q^+_{\ka r}}|v-\hat T_{r}v|^pd\mu\right)^{1/p}\le C\ka^2\left(\dashint_{Q_r^+}|v|^pd\mu\right)^{1/p}.
\end{align}
It is easily seen that for any $q\in\mathbb{A}_1^{r}$, $v-q$ satisfies \eqref{eq:hom-repla-pde} and $\hat T_{r}q=q$. This allows us to replace $v$ with $v-q$ in \eqref{eq:hom-repla-poly-est} to have
\begin{align}\label{eq:hom-poly-diff-est}
\left(\dashint_{Q^+_{\ka r}}|v-\hat q_{r}|^pd\mu\right)^{1/p}\le C\ka^2\left(\dashint_{Q_r^+}|v-q|^pd\mu\right)^{1/p}.
\end{align}
Moreover, it is easily seen that
\begin{align*}
    &\left|((\bar A^{nn})^{\ka r})^{-1}(\bar g^n)^{\ka r}-((\bar A^{nn})^{r})^{-1}(\bar g^n)^{r}\right|\\
    &\lesssim |(\bar g^n)^{\ka r}-(\bar g^n)^{r}|+\|g^n\|_\infty|(\bar A^{nn})^{\ka r}-(\bar A^{nn})^{r}|.
\end{align*}
By recalling $\eta=(\|\D u\|_{L^\infty(Q_2^+)}+\|\bg\|_{L^\infty(Q_4^+)})\eta_A+\eta_\bg$ and applying Lemma~\ref{lem:mean-diff-est}, we deduce
\begin{align}\label{eq:u_o-diff-est}
\left(\dashint_{Q_{\ka r}^+}|u_0^{\ka r}-u_0^r|^pd\mu\right)^{1/p}\le \dashint_{Q_{\ka r}^+}|u_0^{\ka r}-u_0^r|d\mu\le \frac{Cr\eta(r)}{\ka^{n+1}}.
\end{align}
Next, similarly to $\hat q_r$, we define
$$
\hat q_{\ka r}(x):=v(0)+\D_{x'}v(0)\cdot x'-\int_0^{x_n}((\bar A^{nn})^{\ka r}(s))^{-1}\sum_{\de=1}^{n-1}(\bar A^{n\de})^{\ka r}(s)D_\de v(0)ds\in\mathbb{A}_1^{\ka r}.
$$
We recall $\hat q_{r}=\hat T_{r}v\in\mathbb{A}_1^{r}$ (see \eqref{eq:Tv}) and apply Lemma~\ref{lem:mean-diff-est} to get
$$
\left(\dashint_{Q_{\ka r}^+}|\hat q_{\ka r}-\hat q_{r}|^pd\mu\right)^{1/p}\le \frac{Cr\eta_A(r)}{\ka^{n+1}}|\D_{x'}v(0)|.
$$
By \cite{DonPha23}*{Proposition~4.4} (see also \cite{BeDo25}), \eqref{eq:w-est} and the equality $v=u-u_0^{r}-w$, we get
\begin{align*}
|\D v(0)|\le C\left(\dashint_{Q_r^+}|\D v|^pd\mu\right)^{1/p}\le C(\|\D u\|_\infty+\|\bg\|_\infty).
\end{align*}
Thus we have
\begin{align}
    \label{eq:q-diff-est}
    \left(\dashint_{Q_{\ka r}^+}|\hat q_{\ka r}-\hat q_{r}|^pd\mu\right)^{1/p}\le\frac{Cr\eta(r)}{\ka^n}.
\end{align}

Now we denote by $C_\ka$ a constant which may vary from line to line but depends only on $\ka$, $n$, $m$, $\al$, $p$, and $\la$. From the equality $u-u_0^{r}=v+w$ and estimates \eqref{eq:w-est-2}, \eqref{eq:hom-poly-diff-est}, \eqref{eq:u_o-diff-est}, and \eqref{eq:q-diff-est}, we obtain that for any $q\in\mathbb{A}_1^{r}$
\begin{align*}
    &\dashint_{Q^+_{\ka r}}|u-u_0^{\ka r}-\hat q_{\ka r}|^pd\mu\le \dashint_{Q_{\ka r}^+}|u-u_0^{r}-\hat q_{r}|^pd\mu+C_\ka (r\eta(r))^p\\
    &\le \dashint_{Q^+_{\ka r}}|v-\hat q_{r}|^pd\mu+\dashint_{Q^+_{\ka r}}|w|^pd\mu+C_\ka (r\eta(r))^p\\
    &\le C\ka^{2p}\dashint_{Q_r^+}|v-q|^pd\mu+\dashint_{Q^+_{\ka r}}|w|^pd\mu+C_\ka (r\eta(r))^p\\
    &\le C\ka^{2p}\dashint_{Q_r^+}|u-u_0^{r}-q|^pd\mu+C_\ka (r\eta(r))^p.
\end{align*}
This implies \eqref{eq:vp-bdry-est}.
\end{proof}

\begin{lemma}
    \label{lem:vp-int-est}
Let $p,\be,\be'$, and $\ka$ be as in Lemm~\ref{lem:vp-bdry-est}. For any $X_0\in Q^+_3$ and $0<r<d/4$, it holds
    $$
    \vp(X_0,\ka r)\le \ka^{1+\be'}\vp(X_0,r)+Cr\eta(r).
    $$
\end{lemma}

\begin{proof}
For the proof of this interior case, we can follow the argument in the boundary case, Lemma~\ref{lem:vp-bdry-est}, with straight modifications. The only nontrivial part is the estimate \eqref{eq:hom-poly-diff-est} which relies on the fact that $V=0$ on $Q'_{r/2}$. Its interior counterpart can be inferred from \cite{DonXu21}*{(5.10)}, since the condition $r<d/4$ implies that $u$ solves the uniformly parabolic equation \eqref{eq:unif-elliptic} in $Q_r(X_0)$.
\end{proof}

To derive the decay estimate of $\vp$ by combining the previous two lemmas, we need the following result.

\begin{lemma}
    \label{lem:l-i-bound}
Let $\sigma:[0,1]\to[0,\infty)$ be a Dini function satisfying \eqref{eq:alm-mon} and $\sigma\ge\eta$, where $\eta$ is as in \eqref{eq:Dini-tilde}. Suppose that for some point $X_0\in \overline{Q_3^+}$,
\begin{align}\label{eq:vp-bound-assump}
\vp(X_0,r)\le r\sigma(r),\quad 0<r<1/8.
\end{align}
Then there is a constant $C>0$, depending only on $n,m,\la,\al,p,$ such that for every $0<r<1/8$,
\begin{align}
    \label{eq:u-poly-grad-est}
    \left(\dashint_{Q_r^+(X_0)}|\D(u-u_0^{X_0,r}-q^{X_0,r})|^pd\mu\right)^{1/p}\le C\sigma(r)
\end{align}
and  for every $1\le\delta\le n-1$,
\begin{align}\label{eq:l_i-est}
    |\ell_\de^{X_0,r}-\ell_\de^{X_0,2r}|\le C\sigma(r)\quad\text{and}\quad|\ell_\de^{X_0,r}|\le C\left(\|\D u\|_\infty+\|\bg\|_\infty+\int_0^1\frac{\sigma(\rho)}\rho d\rho\right).
\end{align}
\end{lemma}

\begin{proof}
We observe that $u^{X_0,r}:=u-u_0^{X_0,r}-q^{X_0,r}$ satisfies in $Q^+_{2r}(X_0)$
$$
D_\g(x_n^\al(\bar A^{\g\de})^{X_0,r}D_\de u^{X_0,r})-x_n^\al\partial_tu^{X_0,r}=D_\g(x_n^\al((\bar A^{\g\de})^{X_0,r}-A^{\g\de})D_\de u+g^\g-(\bar g^\g)^{X_0,r}))
$$
with the conormal boundary condition
$$
\lim_{x_n\to0+}x_n^\al((\bar A^{n\de})^{X_0,r}D_\de u^{X_0,r}+(A^{n\de}-(\bar A^{n\de})^{X_0,r})D_\de u+(\bar g^\g)^{X_0,r}-g^\g)=0\quad\text{on }Q'_{2r}(X'_0)
$$
if $Q_{2r}(X_0)\cap Q'_4$ is nonempty. We then decompose $u^{X_0,r}=w^{X_0,r}+v^{X_0,r}$ in a similar way as in the proof of Lemma~\ref{lem:psi-bdry-est}. Regarding $w^{X_0,r}$, by applying the weak type-(1,1) estimates in Lemma~\ref{lem:weak-type-(1,1)}, we get
\begin{align}
    \label{eq:w-X_0-est-1}\left(\dashint_{Q^+_{r}(X_0)}| w^{X_0,r}|^pd\mu\right)^{1/p}\le Cr\eta(r),\quad \left(\dashint_{Q^+_{r}(X_0)}|\D w^{X_0,r}|^pd\mu\right)^{1/p}\le C\eta(r).
\end{align}
In addition, by using \cite{DonPha23}*{Lemma~4.2} (see also \cite{BeDo25}) with iteration, we obtain
\begin{align*}
    \left(\dashint_{Q^+_{r}(X_0)}|\D v^{X_0,r}|^pd\mu\right)^{1/p}&\lesssim \frac{1}{r}\left(\dashint_{Q^+_{\frac32r}(X_0)}|v^{X_0,r}|^pd\mu\right)^{1/p}\lesssim \sigma(r),
\end{align*}
where we used \eqref{eq:vp-bound-assump} and \eqref{eq:w-X_0-est-1} with the identity $v^{X_0,r}=u^{X_0,r}-w^{X_0,r}$ in the second step. This, along with \eqref{eq:w-X_0-est-1}, gives \eqref{eq:u-poly-grad-est}.

Moreover, by applying Lemma~\ref{lem:mean-diff-est}, we get
$$
\left(\dashint_{Q_r^+(X_0)}|u_0^{X_0,r}-u_0^{X_0,2r}|^pd\mu\right)^{1/p}\lesssim r\eta(r).
$$
This, together with \eqref{eq:vp-bound-assump} and the assumption $\sigma\ge\eta$, yields
\begin{align}\label{eq:q-l-est-1}\begin{split}
    &\left(\dashint_{Q_r^+(X_0)}|q^{X_0,r}-q^{X_0,2r}|^pd\mu\right)^{1/p}\\
    &\lesssim \left(\dashint_{Q_r^+(X_0)}|u-u_0^{X_0,r}-q^{X_0,r}|^pd\mu\right)^{1/p}+\left(\dashint_{Q_r^+(X_0)}|u-u_0^{X_0,2r}-q^{X_0,2r}|^pd\mu\right)^{1/p}\\
    &\qquad+\left(\dashint_{Q_r^+(X_0)}|u_0^{X_0,r}-u_0^{X_0,2r}|^pd\mu\right)^{1/p}\\
    &\lesssim \vp(X_0,r)+\vp(X_0,2r)+r\eta(r)\lesssim r\sigma(r).
\end{split}\end{align}
We assume without loss of generality $X_0'=0$ and observe that
$$
2(\ell_1^{X_0,r}-\ell_1^{X_0,2r})x_1=(q^{X_0,r}-q^{X_0,2r})(x)-(q^{X_0,r}-q^{X_0,2r})(-x_1,x_2,\ldots,x_n),
$$
which combined with \eqref{eq:q-l-est-1} produces
\begin{align*}
    r|\ell_1^{X_0,r}-\ell_1^{X_0,2r}|\lesssim \left(\dashint_{Q_r^+(X_0)}|q^{X_0,r}-q^{X_0,2r}|^pd\mu\right)^{1/p}\lesssim r\sigma(r).
\end{align*}
By repeating the above process for $2\le \de\le n-1$, we get
\begin{align*}
    |\ell_\de^{X_0,r}-\ell_\de^{X_0,2r}|\lesssim \sigma(r),\quad 1\le \de\le n-1.
\end{align*}

Finally, we take $k\in\mathbb{N}$ satisfying $1/8\le 2^kr<1/4$. Then we have by using the above estimate and \eqref{eq:u-poly-grad-est} 
\begin{equation*}
    |\ell_\delta^{X_0,r}|\le \sum_{i=0}^{k-1}|\ell_\delta^{X_0,2^ir}-\ell_\delta^{X_0,2^{i+1}r}|+|\ell_\delta^{X_0,2^kr}|\lesssim \int_0^1\frac{\sigma(\rho)}\rho d\rho+\|\D u\|_\infty+\|\bg\|_\infty.\qedhere
\end{equation*}
\end{proof}

Next, we combine Lemmas~\ref{lem:vp-bdry-est} - \ref{lem:l-i-bound} to estimate $\vp$.
\begin{lemma}
    \label{lem:vp-est}
    Let $\ka\in(0,1/4)$ be as in Lemma~\ref{lem:vp-bdry-est}. If $X_0\in Q_3^+$ and $0<r<\ka/20$, then
    $$
    \vp(X_0,r)\le Cr\hat\eta(r).
    $$
\end{lemma}

\begin{proof}
Without loss of generality, we may assume $X'_0=0$, i.e., $X_0=(0,d)$. We write for simplicity $q^r=q^{0,r}$, $(\bar g^n)^r=(\bar g^n)^{0,r}$, $\ell_\de^{r}=\ell_\de^{0,r}$ and $(\bar A^{n\de})^r=(\bar A^{n\de})^{0,r}$ for $1\le \de\le n$. 

We first claim that for any $0<R<1/4$
\begin{align}
    \label{eq:vp-bdry-est-2}
    \vp(0,R)\lesssim R\tilde\eta(R).
\end{align}
Indeed, we take a nonnegative integer $j\ge0$ such that $\ka^{-j}R\le1/4<\ka^{-(j+1)}R$. From \eqref{eq:vp-bdry-est} with iteration, we find
\begin{align*}
    \vp(0,R)&\le \ka^{(1+\be')j}\vp(0,\ka^{-j}R)+C\sum_{i=1}^j\ka^{(1+\be')(i-1)}\ka^{-i}R\eta(\ka^{-i}R)\\
    &\lesssim R^{1+\be'}\left(\dashint_{Q_{\ka^{-j}R}}|u-u_0^{X_0,r}|^pd\mu\right)^{1/p}+R\sum_{i=1}^j\ka^{\be' j}\eta(\ka^{-i}R)\\
    &\lesssim (\|u\|_{1,\mu}+\|\bg\|_\infty)R^{1+\be'}+R\tilde\eta(R)\\
    &\lesssim R\tilde\eta(R).
\end{align*}
We then consider two cases either $r\ge \ka d/4$ or $0<r<\ka d/4$.

\medskip\noindent\emph{Case 1.} Suppose $r\ge\ka d/4$. Then, for $R:=(1+4/\ka)r$, \eqref{eq:vp-bdry-est-2} gives
$$
\vp(0,R)\lesssim R\tilde\eta(R)\lesssim r\tilde\eta(r).
$$
Thus we have by Lemmas~\ref{lem:gradient-unif-est} and \ref{lem:l-i-bound} that
\begin{align}
    \label{eq:l-est}
    |\ell_\delta^{5\rho}|\lesssim \|u\|_{1,\mu}+\|\bg\|_\infty+\int_0^1\frac{\eta_\bg(\rho)}\rho d\rho,\quad 1\le\delta\le n-1.
\end{align}
In addition, by the minimality of $\vp$ and the identities $(\bar g^n)^{X_0,r}=(\bar g^n)^r$ and $(\bar A^{n\de})^{X_0,r}=(\bar A^{n\de})^r$, we get
\begin{align*}
    &\vp(X_0,r)\\
    &\le \left(\dashint_{Q^+_r(X_0)}\left|u-\ell_0^{R}-\sum_{\de=1}^{n-1}\ell_\de^{R}x_\de-\int_0^{x_n}((\bar A^{nn})^{r})^{-1}\left((\bar g^n)^{r}+\sum_{\de=1}^{n-1}(\bar A^{n\de})^{r}\ell_\de^{R}\right)\right|^pd\mu\right)^{1/p}\\
    &\lesssim \left(\dashint_{Q^+_r(X_0)}\left|u-\ell_0^{R}-\sum_{\de=1}^{n-1}\ell_\de^{R}x_\de-\int_0^{x_n}((\bar A^{nn})^{R})^{-1}\left((\bar g^n)^{R}+\sum_{\de=1}^{n-1}(\bar A^{n\de})^{R}\ell_\de^{R}\right)\right|^pd\mu\right)^{1/p}\\
    &\qquad+\dashint_{Q^+_r(X_0)}\left(\int_0^{x_n}((\bar A^{nn})^{r})^{-1}\left|(\bar g^n)^{r}-(\bar g^n)^{R}+\sum_{\de=1}^{n-1}((\bar A^{n\de})^{r}-(\bar A^{n\de})^{R})\ell_\de^{R}\right|\right)d\mu\\
    &\qquad+\dashint_{Q^+_r(X_0)}\left(\int_0^{x_n}\left|((\bar A^{nn})^{r})^{-1}-((\bar A^{nn})^{R})^{-1}\right|\left|(\bar g^n)^{R}+\sum_{\de=1}^{n-1}(\bar A^{n\de})^{R}\ell_\de^{R}\right|\right)d\mu\\
    &\lesssim \vp(0,R)+r\hat\eta(r)\lesssim r\hat\eta(r),
\end{align*}
where we used Lemma~\ref{lem:mean-diff-est}, \eqref{eq:vp-bdry-est-2}, and \eqref{eq:l-est} in the last step.

\medskip\noindent\emph{Case 2.} Suppose $0<r<\ka d/4$. Take $j\in \mathbb{N}$ such that $\ka d/4\le \ka^{-j}r<d/4$. By applying Lemma~\ref{lem:vp-int-est} with iteration, we have
\begin{align*}
    \vp(X_0,r)&\lesssim \ka^{(1+\be')j}\vp(X_0,\ka^{-j}r)+r\tilde\eta(r)\lesssim (r/d)^{1+\be'}\vp(X_0,\ka^{-j}r)+r\tilde\eta(r).
\end{align*}
By using the result in Case 1 and the fact that $\rho\longmapsto \rho^{-\be'}\hat\eta(\rho)$ is nonincreasing, we obtain
\begin{equation*}
    \vp(X_0,r)\lesssim (r/d)^{1+\be'}d\hat\eta(d)+r\tilde\eta(r)\lesssim r\hat\eta(r).\qedhere
\end{equation*}
\end{proof}

Now, we are ready to prove Theorem~\ref{thm:sol-time-est}.

\begin{proof}[Proof of Theorem~\ref{thm:sol-time-est}]
We split the proof into three steps.

\medskip\noindent\emph{Step 1.} Recall that $U=A^{n\de}D_\de u-g^n$. The aim of this step is to show that
\begin{align}
        &\left(\dashint_{Q_r^+(X_0)}|D_\de u-\ell_\de^{X_0,r}|^pd\mu\right)^{1/p}\le C\hat\eta(r),\quad 1\le \de\le n-1,\label{eq:u-poly-approx-est-1}\\
        &\left(\dashint_{Q^+_r(X_0)}|U-(d/x_n)^\al\ell_n^{X_0,r}|^pd\mu\right)^{1/p}\le C\hat\eta(r).\label{eq:u-poly-approx-est-2}
\end{align}
Indeed, since $D_\de(u-u_0^{X_0,r}-q^{X_0,r})=D_\de u-\ell_\de^{X_0,r}$ for $1\le \de\le n-1$, \eqref{eq:u-poly-approx-est-1} follows from \eqref{eq:u-poly-grad-est} in Lemma~\ref{lem:l-i-bound}, which is applicable due to Lemma~\ref{lem:vp-est}. Regarding \eqref{eq:u-poly-approx-est-2}, when $r\ge d/4$, we have $(d/x_n)^\al\ell_n^{X_0,r}=0$ and thus it simply follows from Lemma~\ref{lem:psi-est}. On the other hand, when $0<r< d/4$, we observe
\begin{align*}
    &\sum_{\de=1}^n(\bar A^{n\de})^{X_0,r}D_\de(u-u_0^{X_0,r}-q^{X_0,r})\\
    &=(U-(d/x_n)^\al\ell_n^{X_0,r})+(g^n-(\bar g^n)^{X_0,r})+((\bar A^{n\de})^{X_0,r}-A^{n\de})D_\de u,
\end{align*}
which combined with \eqref{eq:u-poly-grad-est} implies \eqref{eq:u-poly-approx-est-2}.

\medskip\noindent\emph{Step 2.} The purpose of this step is to prove
\begin{align}
    \label{eq:u-l_0-diff-est}
    |u(X_0)-\ell_0^{X_0,r}|\le Cr\bar\eta(r),
\end{align}
where $\bar\eta$ is as in \eqref{eq:Dini-tilde}. Since $\ell_n^{X_0,r}=0$ when $r\ge d/4$ and $4/5<d/x_n< 4/3$ when $r< d/4$, we have by using \eqref{eq:u-poly-approx-est-2}
\begin{align*}
    |\ell_n^{X_0,r}-\ell_n^{X_0,2r}|&\le C\left(\dashint_{Q_r^+(X_0)}|U-(d/x_n)^\al\ell_n^{X_0,r}|^p+|U-(d/x_n)^\al\ell_n^{X_0,2r}|^p d\mu\right)^{1/p}\\
    &\le C\hat\eta(r).
\end{align*}
We then follow the argument at the end of the proof of Lemma~\ref{lem:l-i-bound} to get $|\ell_n^{X_0,r}|\le C\int_0^1\frac{\hat\eta(\rho)}\rho d\rho$. These, along with Lemma~\ref{lem:l-i-bound}, give that for any $1\le \delta\le n$
\begin{align}\label{eq:l-bound}
|\ell_\de^{X_0,r}|\le C\left(\|\D u\|_\infty+\int_0^1\frac{\hat\eta(\rho)}\rho d\rho\right) \quad\text{and}\quad |\ell_\de^{X_0,r}-\ell_\de^{X_0,2r}|\le C\hat\eta(r).
\end{align}
Moreover, from the definition of $q^{X_0,r}$ and $q^{X_0,2r}$, we find
\begin{align*}
    &\ell_0^{X_0,r}-\ell_0^{X_0,2r}\\&=(q^{X_0,r}(x)-q^{X_0,2r}(x))-\sum_{\de=1}^{n-1}(\ell_\de^{X_0,r}-\ell_\de^{X_0,2r})(x_\de-(x_0)_\de)\\
    &\qquad-\int_{d}^{x_n}\Bigg(((\bar A^{nn})^{X_0,r})^{-1}\left((d/s)^\al\ell_n^{X_0,r}-\sum_{\de=1}^{n-1}(\bar A^{n\de})^{X_0,r}\ell_\de^{X_0,r}\right)\\
    &\qquad\qquad-((\bar A^{nn})^{X_0,2r})^{-1}\left((d/s)^\al\ell_n^{X_0,2r}-\sum_{\de=1}^{n-1}(\bar A^{n\de})^{X_0,2r}\ell_\de^{X_0,2r}\right)\Bigg)ds.
\end{align*}
By using Lemma~\ref{lem:mean-diff-est}, \eqref{eq:q-l-est-1}, \eqref{eq:l-bound}, and arguing as before, we get
\begin{align}
    \label{eq:l_0-diff-est}
    |\ell_0^{X_0,r}-\ell_0^{X_0,2r}|\le Cr\hat\eta(r).
\end{align}
In addition, since $(u-u_0^{X_0,r}-q^{X_0,r})(X_0)=u(X_0)-\ell_0^{X_0,r}$, we infer from Lemma~\ref{lem:vp-est} that $\ell_0^{X_0,r}\to u(X_0)$ as $r\to0$. This, together with \eqref{eq:l_0-diff-est}, gives \eqref{eq:u-l_0-diff-est}.

\medskip\noindent\emph{Step 3.} In this step, we derive \eqref{eq:sol-time-est}. It is sufficient to show that 
\begin{align}
    |u(t_0,x_0)-u(t_0-r^2,x_0)|\le Cr\bar\eta(r),\quad (x_0,t_0)\in Q_1^+,\,\,\, 0<r<\ka/32,
\end{align}
where $\bar\eta$ is defined in \eqref{eq:Dini-tilde}. To prove it, we write $X_0=(t_0,x_0)$ and $X_1=(t_0-r^2,x_0)$ and assume without loss of generality $x'_0=0$. Due to \eqref{eq:u-l_0-diff-est} and triangle inequality, it suffices to show
\begin{align*}
|\ell_0^{X_0,2r}-\ell_0^{X_1,2r}|\le Cr\bar\eta(r).
\end{align*}
We claim that for any $0<s<1/4$ and $1\le\de\le n$
\begin{align}
    \label{eq:a-g-diff-center-est}&\dashint_{Q_r^+(X_1)}\int_d^{x_n}|(\bar A^{n\de})^{X_0,2r}(s)-(\bar A^{n\de})^{X_1,2r}(s)|dsd\mu\le Cr\bar\eta(r),\\
    \label{eq:l-diff-center-est}&|\ell_\de^{X_0,2r}|+|\ell_\de^{X_1,2r}|\le C\left(|\D u\|_\infty+\int_0^1\frac{\hat\eta(\rho)}\rho d\rho\right),\,\,\,\,\, |\ell_\de^{X_0,2r}-\ell_\de^{X_1,2r}|\le C\bar\eta(r).
\end{align}
Indeed, by using the triangle inequality $|(\bar A^{n\de})^{X_0,2r}-(\bar A^{n\de})^{X_1,2r}|\le |(\bar A^{n\de})^{X_0,2r}-(\bar A^{n\de})^{X_0,4r}|+|(\bar A^{n\de})^{X_0,4r}-(\bar A^{n\de})^{X_1,2r}|$ and the inclusion $Q^+_{4r}(X_1)\subset Q^+_{8r}(X_0)$ and arguing as in the proof of Lemma~\ref{lem:mean-diff-est}, we obtain \eqref{eq:a-g-diff-center-est}. Moreover, the first estimate in \eqref{eq:l-diff-center-est} follows from \eqref{eq:l-bound}. For the second one, by using \eqref{eq:u-poly-approx-est-1} and arguing as before, we obtain $|\ell_\de^{X_0,2r}-\ell_\de^{X_1,2r}|\lesssim \bar\eta(r)$, $1\le \de\le n-1$. Regarding the case $\de=n$, we use $(x_0)_n=(x_1)_n$ and apply \eqref{eq:u-poly-approx-est-2} to get
\begin{align*}
    &|\ell_n^{X_0,2r}-\ell_n^{X_1,2r}|\\
    &\lesssim \left(\dashint_{Q_r^+(X_1)}|U-((x_0)_n/x_n)^\al\ell_n^{X_0,2r}|^p+|U-((x_1)_n/x_n)^\al\ell_n^{X_1,2r}|^pd\mu\right)^{1/p}\\
    &\lesssim \left(\dashint_{Q_{2r}^+(X_0)}|U-((x_0)_n/x_n)^\al\ell_n^{X_0,2r}|^pd\mu\right)^{1/p}\\
    &\qquad+\left(\dashint_{Q_{2r}^+(X_1)}|U-((x_1)_n/x_n)^\al\ell_n^{X_1,2r}|^pd\mu\right)^{1/p}\\
    &\lesssim \bar\eta(r).
\end{align*}
Now, we observe that
\begin{align*}
    &|\ell_0^{X_0,2r}-\ell_0^{X_1,2r}|\\
    &\le \left(\dashint_{Q_r^+(X_1)}|q^{X_0,2r}-q^{X_1,2r}|^pd\mu\right)^{1/p}+\left(\dashint_{Q_r^+(X_1)}\sum_{\de=1}^{n-1}|(\ell_\de^{X_0,2r}-\ell_\de^{X_1,2r})x_i|^pd\mu\right)^{1/p}\\
    &\qquad+\Bigg(\dashint_{Q_r^+(X_1)}\bigg|\int_{d}^{x_n}\bigg[((\bar A^{nn})^{X_0,2r})^{-1}\bigg((d/s)^\al\ell_n^{X_0,2r}-\sum_{\de=1}^{n-1}(\bar A^{n\de})^{X_0,2r}\ell_\de^{X_0,2r}\bigg)\\
    &\quad\qquad\qquad\quad -((\bar A^{nn})^{X_1,2r})^{-1}\bigg((d/s)^\al\ell_n^{X_1,2r}-\sum_{\de=1}^{n-1}(\bar A^{n\de})^{X_1,2r}\ell_\de^{X_1,2r}\bigg)\bigg]ds\bigg|^pd\mu\bigg)^{1/p}\\
    &=:I+II+III.
\end{align*}
By using \eqref{eq:a-g-diff-center-est} and \eqref{eq:l-diff-center-est}, we
get $II+III\le Cr\bar\eta(r)$. Concerning $I$, since \eqref{eq:a-g-diff-center-est} also holds when $\bar A^{n\de}$ is replaced by $\bar g^n$, we have 
$$
\left(\dashint_{Q_r^+(X_1)}|u_0^{X_0,2r}-u_0^{X_1,2r}|^pd\mu\right)^{1/p}\le Cr\bar\eta(r).$$ 
By combining this with Lemma~\ref{lem:vp-est} and the triangle inequality 
$$
|q^{X_0,2r}-q^{X_1,2r}|\le |u-u_0^{X_0,2r}-q^{X_0,2r}|+|u-u_0^{X_1,2r}-q^{X_1,2r}|+|u_0^{X_0,2r}-u_0^{X_1,2r}|,
$$
we obtain $I\le Cr\bar\eta(r)$. This finishes the proof.
\end{proof}

Before closing this section, we prove the continuity of $\frac{\delta_{t,h}u}{h^{1/2}}$.

\begin{theorem}
    \label{thm:diff-quo-conti}
    Let $u$, $A$, and $\bg$ be as in Theorem~\ref{thm:DMO-reg} and $0<\be<1$. Then $\frac{\delta_{t,h}u}{h^{1/2}}\in C_X(\overline{Q_1^+})$ uniformly in $h\in(0,1/4)$.
\end{theorem}

\begin{proof}
We fix $0<\rho<1$, $1\le\g\le n$ and $(t,x)\in \overline{Q_1^+}$, and let
$$
U_{\g,\rho}:=\frac{\delta_{t,h}u(t,x)-\delta_{t,h}u(t,x-\rho e_\g)}{h^{1/2}}.
$$
Suppose $h<\rho^2$. Then $\frac{|\delta_{t,h}u(t,x)|}{h^{1/2}}\lesssim \omega_t(h)$ by Theorem~\ref{thm:sol-time-est}. This, along with Lemma~\ref{lem:Dini-sum-est}, gives
\begin{align}
    \label{eq:U-diff-quo-est}
    |U_{\g,\rho}|\lesssim\omega_t(h)=\sum_{k=0}^\infty\frac1{2^k}\hat\eta(h^{1/2}/{2^k})\lesssim\sum_{k=0}^\infty\frac1{2^k}\int_0^{\frac{\rho}{2^k}}\frac{\hat\eta(s)}sds\lesssim \int_0^\rho\frac{\hat\eta(s)}sds\le \sigma_0(\rho).
\end{align}
On the other hand, when $h\ge\rho^2$, we have by \eqref{eq:mod-conti-space-est} that for $\omega_x$ as in \eqref{eq:omega-x}
\begin{align*}
    |U_{\g,\rho}|&=\frac{|u(t,x)-u(t,x-\rho e_\g)-(u(t-h,x)-u(t-h,x-\rho e_\g)|}{h^{1/2}}\\
    &\le \frac{\rho}{h^{1/2}}\int_0^1|D_\g u(t,x-\tau\rho e_\g)-D_\g u(t-h,x-\tau\rho e_\g)|d\tau\\
    &\lesssim\frac{\rho\,\omega_x(h^{1/2})}{h^{1/2}}\lesssim \omega_x(\rho),
\end{align*}
where in the last step we used the monotonicity of $r\longmapsto\frac{\omega_x(r)}r$, which follows from that of $r\longmapsto\frac{\hat\eta_\bullet(r)}r$. Therefore, we conclude that $\frac{\delta_{t,h}u}{h^{1/2}}\in C_x(\overline{Q_1^+})$ uniformly in $h$.

Next, we prove $\frac{\delta_{t,h}u}{h^{1/2}}\in C_t(\overline{Q_1^+})$ uniformly in $h$. Set
$$
U_{t,\rho}:=\frac{\delta_{t,h}u(t,x)-\delta_{t,h}u(t-\rho,x)}{h^{1/2}}.
$$
If $h<\rho$, then we apply Theorem~\ref{thm:sol-time-est} to have $|U_{t,\rho}|\lesssim\omega_t(h)$ and argue as in \eqref{eq:U-diff-quo-est} to obtain $|U_{t,\rho}|\lesssim\sigma_0(\rho^{1/2})$. On the other hand, when $h\ge\rho$, we use Theorem~\ref{thm:sol-time-est} again to deduce
\begin{align*}
    |U_{t,\rho}|&\le \frac{|u(t,x)-u(t-\rho,x)|+|u(t-h,x)-u(t-h-\rho,x)|}{h^{1/2}}\lesssim \frac{\rho^{1/2}\omega_t(\rho)}{h^{1/2}}\le\omega_t(\rho).
\end{align*}
\end{proof}


\section{Higher order Schauder type Estimates}\label{sec:HO-schauder}
The purpose of this section is to provide the proof of Theorem~\ref{thm:DMO-HO-reg-par}. We begin with a simpler case when the coefficients are in Hölder spaces and then use that result to establish the theorem.

\subsection{Systems with Hölder coefficients}

This section is devoted to the proof of the following result.

\begin{theorem}
    \label{thm:holder-HO-reg}
    Let $k\in\mathbb{N}$ and $0<\be<1$ and let $u$ be a solution of \eqref{eq:pde-par}.
    Suppose $A$ satisfies \eqref{eq:assump-coeffi} and $A,\bg\in C^{\frac{k+\be-1}2,k+\be-1}(\overline{Q_1^+})$. Then $u\in C^{\frac{k+\be}2,k+\be}(\overline{Q^+_{1/2}})$.
\end{theorem}

The case when $k=1$ in Theorem~\ref{thm:holder-HO-reg} can be inferred from the proofs of Theorems~\ref{thm:DMO-reg}, \ref{thm:sol-time-est}, and \ref{thm:diff-quo-conti}.

The following theorem addresses the case when $k=2$. This is derived from an application of the result when $k=1$ and a fine analysis of finite difference quotients.

\begin{theorem}
    \label{thm:holder-HO-reg-2}
    Theorem~\ref{thm:holder-HO-reg} holds when $k=2$.
\end{theorem}

\begin{proof}
To show that $u\in C^{\frac{2+\be}2,2+\be}$, it is enough to show
\begin{align}
    \label{eq:holder-cond-2}
    D_{x'}u\in C^{\frac{1+\be}2,1+\be},\quad \partial_tu\in C^{\be/2,\be}, \quad D_nu\in C^{\frac{1+\be}2}_t \quad\text{and}\quad D_{nn}u\in C^{\be/2,\be}.
\end{align}
We divide the proof of \eqref{eq:holder-cond-2} into three steps.

\medskip\noindent\emph{Step 1.} We first prove $D_{x'}u\in C^{\frac{1+\be}2,1+\be}$. From \eqref{eq:pde-par}, we see that $D_{x'}u$ solves
\begin{align}\label{eq:pde-tan-deriv}
    \begin{cases}
        D_\g(x_n^\al A^{\g\de}D_\de(D_{x'}u))-x_n^\al\partial_t(D_{x'}u)=\div(x_n^\al\bar\bg)\\
        \lim_{x_n\to0+}x_n^\al(A^{n\de}D_\de(D_{x'}u)-\bar g^n)=0
    \end{cases}\text{in }Q_1^+,
\end{align}
where $\bar g^\g:=D_{x'}g^\g-D_{x'}A^{\g\de}D_\de u$, $1\le\g\le n$ (one can derive \eqref{eq:pde-tan-deriv} rigorously by using the finite difference quotient). Since $u\in C^{\frac{1+\be}2,1+\be}$ by the case $k=1$, we have $\bar\bg\in C^{\be/2,\be}$, and thus $D_{x'}u\in C^{\frac{1+\be}2,1+\be}$ by the case $k=1$ again.

\medskip\noindent\emph{Step 2.} We fix $h\in(0,1/4)$ and write
\begin{align}\label{eq:diff-quo}
u^h(t,x):=\frac{u(t,x)-u(t-h,x)}{h^{1/2}}=\frac{\de_{t,h}u(t,x)}{h^{1/2}}.
\end{align}
As in Section~\ref{subsec:reg-time}, we write $\|u\|_{L^1(Q_1^+,d\mu)}=\|u\|_{1,\mu}$, where $d\mu=x_n^{\al^-}dX$. The aim of this step is to prove that $u^h\in C^{\frac{1+\be}2,1+\be}(\overline{Q_{1/2}^+})$ with 
\begin{align}
    \label{eq:u-h-holder}
    \|u^h\|_{C^{\frac{1+\be}2,1+\be}(\overline{Q^+_{1/2}})}\le C(\|u\|_{1,\mu}+\|\bg\|_{C^{\frac{1+\be}2,1+\be}})
\end{align}
for some constant $C>0$, independent of $h\in (0,1/4)$. To this end, we define $(A^{\g\de})^h$ and $\bg^h$ in a similar way as \eqref{eq:diff-quo}. From \eqref{eq:pde-par}, we deduce 
\begin{align}\label{eq:pde-diff-quo}
    \begin{cases}
        D_\g(x_n^\al A^{\g\de}D_\de u^h)-x_n^\al\partial_t u^h=\div(x_n^\al\hat\bg^h)\\
        \lim_{x_n\to 0+}x_n^\al(A^{n\de}D_\de u^h-(\hat g^n)^h)=0
    \end{cases}\text{in }Q_1^+,
\end{align}
where $(\hat g^\g)^h(t,x):=(g^\g)^h(t,x)-(A^{\g\de})^h(t,x)D_\de u(t-h,x)$, $1\le\g\le n$.

We claim that there exists a constant $C>0$, independent of $h$, such that
\begin{align}
    \label{eq:diff-quo-holder}\begin{split}
    &|\bg^h(t,x)-\bg^{h}(s,x)|\le C[\bg]_{C_t^{\frac{1+\be}2}}|t-s|^{\be/2}\quad\text{for any $t,s\in (-1,0)$ and $x\in D_1^+$},\\
    &|\bg^h(t,x)-\bg^{h}(t,y)|\le C[\bg]_{C^{\frac{1+\be}2, 1+\be}}|x-y|^{\be}\quad\text{for any $t\in (-1,0)$ and $x,y\in D_1^+$}.
\end{split}\end{align}
We first prove the first estimate in \eqref{eq:diff-quo-holder}. If $|t-s|\le h$, then
\begin{align*}
    |\bg^h(t,x)-\bg^h(s,x)|&\le \frac{|\bg(t,x)-\bg(s,x)|+|\bg(t-h,x)-\bg(s-h,x)|}{h^{1/2}}\\
    &\le \frac{C[\bg]_{C_t^{\frac{1+\be}2}}|t-s|^{\frac{1+\be}2}}{h^{1/2}}\le C[g]_{C_t^{\frac{1+\be}2}}|t-s|^{\be/2}.
\end{align*}
On the other hand, if $|t-s|>h$, then from
\begin{align}\label{eq:g-h-bound}
|\bg^h(t,x)|=\frac{|\bg(t,x)-\bg(t-h,x)|}{h^{1/2}}\le C[\bg]_{C_t^{\frac{1+\be}2}}h^{\be/2},
\end{align}
we have
$$
|\bg^h(t,x)-\bg^{h}(s,x)|\le |\bg^h(t,x)|+|\bg^{h}(s,x)|\le C[\bg]_{C_t^{\frac{1+\be}2}}|t-s|^{\be/2}.
$$
Next, we verify the second estimate in \eqref{eq:diff-quo-holder}. If $|x-y|\le h^{1/2}$, then we have
\begin{align*}
    &|\bg^h(t,x)-\bg^h(t,y)|\\
    &=\frac{|(\bg(t,x)-\bg(t,y))-(\bg(t-h,x)-\bg(t-h,y))|}{h^{1/2}}\\
    &\le \frac{(|\D\bg(t,z_1)-\D\bg(t,z_2)|+|\D\bg(t,z_2)-\D\bg(t-h,z_2)|)|x-y|}{h^{1/2}}\\
    &\le \frac{|x-y|^{\be+1}[\bg]_{C^{\frac{1+\be}2, 1+\be}}}{h^{1/2}}+\frac{|x-y|[\bg]_{C^{\frac{1+\be}2, 1+\be}}}{h^{\frac{1-\be}2}}\le 2[\bg]_{C^{\frac{1+\be}2, 1+\be}}|x-y|^\be,
\end{align*}
where we applied the mean value theorem and the triangle inequality in the second step.
On the other hand, if $|x-y|>h^{1/2}$, then we use \eqref{eq:g-h-bound} to get
$$
|\bg^h(t,x)-\bg^h(t,y)|\le C[\bg]_{C_t^{\frac{1+\be}2}}h^{\be/2}\le C[\bg]_{C_t^{\frac{1+\be}2}}|x-y|^\be.
$$

Notice that \eqref{eq:diff-quo-holder} and \eqref{eq:g-h-bound} imply $\|\bg^h\|_{C^{\be/2,\be}}\le C[\bg]_{C^{\frac{1+\be}2, 1+\be}}$ for some constant $C>0$, independent of $h$. Since this holds for $(A^{\g\de})^h$ by the same reasoning and also for $D_x u$ as $u\in C^{\frac{1+\be}2,1+\be}$, we have $\|\hat\bg^h\|_{C^{\be/2,\be}}\le C(\|u\|_{1,\mu}+\|\bg\|_{C^{\frac{1+\be}2,1+\be}})$. Now, \eqref{eq:u-h-holder} follows by the induction hypothesis.

\medskip\noindent\emph{Step 3.} In this step, we show that \eqref{eq:u-h-holder} implies $D_nu\in C^{\frac{1+\be}2}_t$ and $\partial_tu\in C^{\be/2}$. Indeed, from \eqref{eq:u-h-holder}, we have $\|D_nu^h\|_{C^{\be/2}_t}\le(\|u\|_{1,\mu}+\|\bg\|_{C^{\frac{1+\be}2,1+\be}})$, which yields
\begin{align*}
    |D_nu(t+h,x)-2D_nu(t,x)+D_nu(t-h,x)|\le C(\|u\|_{1,\mu}+\|\bg\|_{C^{\frac{1+\be}2,1+\be}})h^{\frac{1+\be}2}.
\end{align*}
Thus, $D_nu\in C^{\frac{1+\be}2}_t$ by \cite{Ste70}*{Propositions 8 in Chapter 5}.

Moreover, since \eqref{eq:u-h-holder} gives $u^h\in C^{\frac{1+\be}2}_t$, we have
$$
|u(t+h,x)+u(t-h,x)-2u(t,x)|\le C(\|u\|_{1,\mu}+\|\bg\|_{C^{\frac{1+\be}2,1+\be}})h^{1+\frac\be2}.
$$
Thus $\partial_tu\in C^{\be/2}_t$ by \cite{Ste70}*{Propositions 8 and 9 in Chapter 5}. 

It remains to show $\partial_tu\in C^\be_x$. To this end, we claim that there exists a constant $C>0$, independent of $h$, such that
\begin{align}
    \label{eq:delta-u-h-est}
    |\de_{t,h}\de_{x_\g,r}^2u^h|\le C(\|u\|_{1,\mu}+\|\bg\|_{C^{\frac{1+\be}2,1+\be}})\min\left\{h^{\frac{1+\be}2}, r^{1+\be}\right\}\quad\text{in }Q^+_{3/4}.
\end{align}
Indeed, if $r\le \sqrt h$, then we use $|\de_{x_\g,r}^2u^h|\le C(\|u\|_{1,\mu}+\|\bg\|_{C^{\frac{1+\be}2,1+\be}})r^{1+\be}$, which follows from \eqref{eq:u-h-holder}, to obtain
\begin{align*}
|\delta_{t,h}(\delta_{x_\g,r}^2u^h)(t,x)|&\le |\de_{x_\g,r}^2u^h(t,x)|+|\de_{x_\g,r}^2u^h(t-h,x)|\\
&\le C(\|u\|_{1,\mu}+\|\bg\|_{C^{\frac{1+\be}2,1+\be}})r^{1+\be}.  
\end{align*}
Similarly, if $\sqrt h<r$, then we use $|\de_{t,h}u^h|\le C(\|u\|_{1,\mu}+\|\bg\|_{C^{\frac{1+\be}2,1+\be}})h^{\frac{1+\be}2}$ to get $$
|\delta_{x_\g,r}^2(\de_{t,h}u^h)|\le C(\|u\|_{1,\mu}+\|\bg\|_{C^{\frac{1+\be}2,1+\be}})h^{\frac{1+\be}2}.
$$

We fix $0<r<1/4$, and let $h=h_j:=2^{-j-2}$ with large $j\in\mathbb{N}$ so that $h<r^2/4$. It is directly seen that for any function $f$,
$$
\de_{t,h}f=\frac12\de_{t,2h}f+\frac12\de_{t,h}^2f.
$$
By iterating this equality and putting $f=\de_{x_\g,r}^2u$, $1\le \g\le n$, we deduce
\begin{align}
    \label{eq:delta-u-div}
    \begin{split}
        \de_{t,h}\de_{x_\g,r}^2u&=\frac1{2^j}\de_{t,2^jh}\de_{x_\g,r}^2u+\sum_{i=0}^{j-1}\frac1{2^{i+1}}\de_{t,2^ih}^2\de_{x_\g,r}^2u=:{\rm I}+{\rm II}.
    \end{split}
\end{align}
Regarding $\rm{I}$, we use $u^{2^jh}=\frac{\de_{t,2^jh}u}{\sqrt{2^jh}}$, $2^jh=1/4$, and 
$$
\left[u^{2^{j}h}\right]_{C^{\frac{1+\be}2,1+\be}}\le C(\|u\|_{1,\mu}+\|\bg\|_{C^{\frac{1+\be}2,1+\be}})
$$ 
to get
$$
|{\rm I}|\le\frac{\sqrt{2^jh}}{2^j}\left|\de_{x_\g,r}^2u^{2^jh}\right|\le C(\|u\|_{1,\mu}+\|\bg\|_{C^{\frac{1+\be}2,1+\be}})hr^{1+\be}.
$$
To estimate ${\rm II}$, we take $i_0\in \mathbb{N}$ such that $2^{i_0-1}h<r^2\le 2^{i_0}h$. We note that $i_0<j$ and use $u^{2^ih}=\frac{\de_{t,2^ih}u}{\sqrt{2^ih}}$ to have
\begin{align*}
    {\rm II}&=\sum_{i=0}^{j-1}\frac1{2^{i+1}}\de_{t,2^ih}\de_{x_\g,r}^2(\de_{t,2^ih}u)=\sum_{i=0}^{j-1}\frac{\sqrt{2^ih}}{2^{i+1}}\de_{t,2^ih}\de_{x_\g,r}^2u^{2^ih}\\
    &=\frac12\sum_{i=0}^{i_0-1}\sqrt{\frac{h}{2^i}}\de_{t,2^ih}\de_{x_\g,r}^2u^{2^ih}+\frac12\sum_{i=i_0}^{j-1}\sqrt{\frac{h}{2^i}}\de_{t,2^ih}\de_{x_\g,r}^2u^{2^ih}\\
    &=:{\rm II}_1+{\rm II}_2.
\end{align*}
For ${\rm II}_1$, we use \eqref{eq:delta-u-h-est} and $2^{i_0-1}h<r^2$ to get
\begin{align*}
    |{\rm II}_1|&\le C(\|u\|_{1,\mu}+\|\bg\|_{C^{\frac{1+\be}2,1+\be}})\sum_{i=0}^{i_0-1}\sqrt{\frac{h}{2^i}}(2^ih)^{\frac{1+\be}2}\\
    &\le C(\|u\|_{1,\mu}+\|\bg\|_{C^{\frac{1+\be}2,1+\be}})\sum_{i=0}^{i_0-1}\sqrt{\frac{h}{2^i}}(2^ih)^{1/2}(2^{i-i_0+1}r^2)^{\be/2}\\
    &\le C(\|u\|_{1,\mu}+\|\bg\|_{C^{\frac{1+\be}2,1+\be}})hr^\be\sum_{i=0}^{i_0-1}(2^{i-i_0+1})^{\be/2}\\
    &\le C(\|u\|_{1,\mu}+\|\bg\|_{C^{\frac{1+\be}2,1+\be}})hr^\be.
\end{align*}
Concerning ${\rm II}_2$, from \eqref{eq:delta-u-h-est} and $r^2\le 2^{i_0}h$, we find
\begin{align*}
|{\rm II}_2|&\le C(\|u\|_{1,\mu}+\|\bg\|_{C^{\frac{1+\be}2,1+\be}})\sum_{i=i_0}^{\infty}\sqrt{\frac{h}{2^i}}r^{1+\be}\le C(\|u\|_{1,\mu}+\|\bg\|_{C^{\frac{1+\be}2,1+\be}})\sqrt{\frac{h}{2^{i_0}}}r^{1+\be}\\
&\le C(\|u\|_{1,\mu}+\|\bg\|_{C^{\frac{1+\be}2,1+\be}})hr^\be.
\end{align*}
Now, combining \eqref{eq:delta-u-div} with the estimates for ${\rm I}$, ${\rm II}_1$, and ${\rm II}_2$ yields
\begin{align*}
    \left|\de_{x_\g,r}^2\left(\frac{\de_{t,h}u}h\right)\right|=\frac{|\delta_{t,h}\de_{x_\g,r}^2u|}h\le C(\|u\|_{1,\mu}+\|\bg\|_{C^{\frac{1+\be}2,1+\be}})r^\be
\end{align*}
for some constant $C>0$, independent of $h$. By taking $h=h_j\to0$, we infer
$$
|\de_{x_\g,r}^2u_t|\le C(\|u\|_{1,\mu}+\|\bg\|_{C^{\frac{1+\be}2,1+\be}})r^\be,\quad 1\le \g\le n.
$$
By applying \cite{Ste70}*{Proposition~8 in Chapter 5}, we conclude $u_t\in C^\be_x$.

\medskip\noindent\emph{Step 4.} In this step, we show that $D_{nn}u\in C^{\be/2,\be}$. To this end, we recall $U=A^{n\de}D_\de u-g^n$ and use \eqref{eq:pde-par} to have
$$
\frac\al{x_n}U=\partial_tu-D_\g(A^{\g\de}D_\de u-g^\g).
$$
This yields
$$
x_n^{-\al}D_n(x_n^\al U)=\frac{\al}{x_n}U+D_nU=\xi,\quad\text{where }\xi:=\partial_tu-\sum_{\g=1}^{n-1}D_\g(A^{\g\de}D_\de u-g^\g).
$$
Notice that $\xi\in C^{\be/2,\be}$ by the results in the previous steps. By using the conormal boundary condition $x_n^\al U=0$ on $Q'_1$, we deduce 
$$
U(X',x_n)=\frac1{x_n^\al}\int_0^{x_n}\rho^\al \xi(X',\rho)d\rho,
$$
which gives
\begin{align*}
    D_n(A^{n\de}D_\de u-g^n)=D_nU=\xi(X',x_n)-\al\int_0^1\rho^\al \xi(X',\rho x_n)d\rho\in C^{\be/2,\be}.
\end{align*}
This implies $D_{nn}u\in C^{\be/2,\be}$. 
\end{proof}

We now achieve Theorem~\ref{thm:holder-HO-reg} for any $k\in\mathbb{N}$ by using the cases when $k=1,2$ and an induction argument.

\begin{proof}[Proof of Theorem~\ref{thm:holder-HO-reg}]
We assume that the theorem holds true for $1,2,\ldots,k$ for some $k\ge2$, and prove it for $k+1$. To verify that $u\in C^{\frac{k+\be+1}2,k+\be+1}$, it suffices to show
\begin{align*}
    D_{x'}u\in C^{\frac{k+\be}2,k+\be},\quad \partial_tu\in C^{\frac{k+\be-1}2,k+\be-1}\quad \text{and}\quad D_{nn}u\in C^{\frac{k+\be-1}2,k+\be-1}.
\end{align*}  
Since $D_{x'}u$ satisfies \eqref{eq:pde-tan-deriv} with $\bar g^\g=D_{x'}g^\g-D_{x'}A^{\g\de}D_\de u\in C^{\frac{k+\be-1}2,k+\be-1}$, we have by the induction hypothesis $D_{x'}u\in C^{\frac{k+\be}2,k+\be}$. Similarly, as $u_t$ solves
\begin{align}\label{eq:pde-time-deriv}
    \begin{cases}
        D_\de(x_n^\al A^{\g\de}D_\de u_t)-x_n^\al\partial_t u_t=\div(x_n^\al\tilde\bg)\\
        \lim_{x_n\to0+}x_n^\al(A^{n\de}D_\de u_t-\tilde g^n)=0
    \end{cases}\text{in }Q^+_1,
\end{align}
where $\tilde g^\g:=\partial_t g^\g-\partial_tA^{\g\de}D_\de u\in C^{\frac{k+\be-2}2,k+\be-2}$, by the induction hypothesis for $k-1$, we get $u_t\in C^{\frac{k+\be-1}2,k+\be-1}$. Finally, by following the argument in Step 4 in the proof of Theorem~\ref{thm:holder-HO-reg-2}, we obtain $D_{nn}u\in C^{\frac{k+\be-1}2,k+\be-1}$.
\end{proof}


\subsection{Systems with partially DMO coefficients}

In this section, we complete the proof of our central result on Schauder type estimates, Theorem~\ref{thm:DMO-HO-reg-par}. The case $k=1$ is a consequence of Theorems~\ref{thm:DMO-reg}, \ref{thm:sol-time-est}, and \ref{thm:diff-quo-conti}. Once Theorem~\ref{thm:DMO-HO-reg-par} is established for $k=2$, the remaining case when $k\ge3$ easily follows by induction, as shown in the proof of Theorem~\ref{thm:holder-HO-reg}. We will establish the theorem for $k=2$ by using the results when $k=1$ and difference quotient approximation. 

We define $\eta_1, \tilde\eta_1, \hat\eta_1, \bar\eta_1$, and $\omega_{t,1}$ as in \eqref{eq:Dini-tilde}-\eqref{eq:omega-t}, but with $\eta_A$ and $\eta_\bg$ replaced by $\eta_{A,1}$ and $\eta_{\bg,1}$. Similarly, we let $\omega_{x,1}(r)$ be as in \eqref{eq:omega-x} with $\hat\eta_A$ and $\hat\eta_\bg$ replaced by $\hat\eta_{A,1}$ and $\hat\eta_{\bg,1}$. Note that $\omega_{x,1}(r)\lesssim \sigma_1(r)$ and $\omega_{t,1}(r)\lesssim\sigma_1(r^{1/2})$.

\begin{proof}[Proof of Theorem~\ref{thm:DMO-HO-reg-par}]
As mentioned above, it suffices to show the theorem for $k=2$. For this, we need to show $D_{x'}u\in \mathring{C}^{1/2,1}$, $\partial_tu\in \mathring{C}$, $D_{nn}u\in \mathring{C}$ and 
\begin{align}
    \label{eq:D_nu-condition}
    \lim_{h\to0}\frac{\delta_{t,h}D_nu}{h^{1/2}}(t,x)=0\quad\text{and}\quad\frac{\delta_{t,h}D_nu}{h^{1/2}}\in C_X(\overline{Q_1^+}),
\end{align}
uniformly in $(t,x)$ and $h$, respectively.
We divide the proof of these into three steps.

\medskip\noindent\emph{Step 1.} We first prove $D_{x'}u\in \mathring{C}^{1/2,1}$. Recall that $D_{x'}u$ solves \eqref{eq:pde-tan-deriv} with $\bar g^\g=D_{x'}g^\g-D_{x'}A^{\g\de}D_\de u$. Clearly, $D_{x'}g^\g,D_{x'}A^{\g\de}\in\cH$. Moreover, since $u$ is a solution of \eqref{eq:pde-par} and the data $A$ and $\bg$ belong to $C^{1/2,1}$, we have $u\in C^{\frac{1+\be}2,1+\be}$ for any $0<\be<1$ by Theorem~\ref{thm:holder-HO-reg}. Thus, $\bar\bg\in \cH$ by Lemma~\ref{lem:product-par-DMO-elliptic}, and hence we have $D_{x'}u\in \mathring{C}^{1/2,1}$ by the case when $k=1$.

\medskip\noindent\emph{Step 2.} In this step, we prove $\partial_tu\in \mathring{C}$. To this end, we fix $h\in(0,1/4)$ and define $u^h$, $(A^{\g\de})^h$, and $\bg^h$ as before; see \eqref{eq:diff-quo}. Then $u^h$ solves \eqref{eq:pde-diff-quo} with $(\hat g^\g)^h(t,x)=(g^\g)^h(t,x)-(A^{\g\de})^h(t,x)D_\de u(t-h,x)\in\cH$. By Theorem~\ref{thm:sol-time-est}, we have that
$$
\frac{|u^h(t+h,x)-u^h(t,x)|}{h^{1/2}}\le\omega_{t,1}(h),\quad (t+h,x),(t,x)\in Q_{3/4}^+.
$$
This implies
\begin{align}
    \label{eq:u-diff-quotient-time}
    \frac{|u(t+2h,x)-2u(t+h,x)+u(t,x)|}{h}\le \omega_{t,1}(h),\quad 0<h<1/4.
\end{align}
In the rest of Step 2, we fix a point $x\in D_{1/2}^+$ and write for simplicity $u(t,x)=u(t)$. 
It can be directly checked that for any $j\in \mathbb{N}$
\begin{align}\label{eq:diff-quotient}
    u(t+h)-u(t)=\frac{u(t+2^jh)-u(t)}{2^j}-\sum_{i=1}^j\frac{u(t+2^ih)-2u(t+2^{i-1}h)+u(t)}{2^i}.
\end{align}
We take $j\in \mathbb{N}$ satisfying $1/4< 2^jh\le 1/2$. Then \eqref{eq:u-diff-quotient-time} and \eqref{eq:diff-quotient} yield
\begin{align*}
    \frac{|u(t+h)-u(t)|}h&\le \frac{|u(t+2^jh)|+|u(t)|}{2^jh}+\frac12\sum_{i=1}^j\frac{|u(t+2^ih)-2u(t+2^{i-1}h)+u(t)|}{2^{i-1}h}\\
    &\le 8\|u\|_\infty+\frac12\sum_{i=1}^j\omega_{t,1}(2^{i-1}h)\le 8\|u\|_\infty+C\int_0^1\frac{\omega_{t,1}(s)}sds<\infty.
\end{align*}
Since $0<h<1/4$ is arbitrary, this implies that $u$ is Lipschitz in time.

Next, we improve the Lipschitz regularity of $u$ in time. To this end, let $\xi:\R\to[0,\infty)$ be an even and nonnegative smooth function satisfying $\supp\xi\in(-1,1)$ and $\int_{-1}^1\xi(s)\,ds=1$. We consider the convolution of $u$ in time-variable
$$
u_{(h)}(t,x):=\int_{-1}^{1}u(t+hs,x)\xi(s)ds.
$$
Again, we write for simplicity $u_{(h)}(t,x)=u_{(h)}(t)$. By the properties of $\xi$ and \eqref{eq:u-diff-quotient-time},
\begin{align}
    \label{eq:u-conv-diff-est}\begin{split}
    |u_{(h)}(t)-u(t)|&=\left|\int_0^1(u(t+hs)+u(t-hs))\xi(s)ds-\int_0^12u(t)\xi(s)ds \right|\\
    &\le h\int_0^1s\omega_{t,1}(hs)\xi(s)ds\le Ch\omega_{t,1}^*(h),
\end{split}\end{align}
where $\omega_{t,1}^*(h):=\int_0^1\omega_{t,1}(hs)ds$. Moreover, by using the change of variables, the equality $\int_{-1}^1\xi''(s)ds=0$ and \eqref{eq:u-diff-quotient-time}, we obtain
\begin{align*}
    |\partial_{tt}u_{(h)}(t)|&=\left|\int_{-1}^1u(t+hs)\frac{\xi''(s)}{h^2}ds\right|=\left|\int_{-1}^1\left(u(t+hs)-u(t)\right)\frac{\xi''(s)}{h^2}ds\right|\\
    &\le \int_0^1\left|u(t+hs)+u(t-hs)-2u(t)\right|\frac{|\xi''(s)|}{h^2}ds\le C\frac{\omega_{t,1}^*(h)}{h}.
\end{align*}
Let $P_{t_0,h}(t)=a_{t_0,h}+b_{t_0,h}(t-t_0)$ be the first-order Taylor expansion of $u_{(h)}$ at $t_0$. Notice that $a_{t_0,h}=u_{(h)}(t_0)$, $b_{t_0,h}=\partial_tu_{(h)}(t_0)$ and $|b_{t_0,h}|\le \|\partial_tu\|_\infty$. Then the above estimate on $\partial_{tt}u_{(h)}$ gives
\begin{align}
    \label{eq:u-Tay-diff-est}
    \|u_{(h)}-P_{t_0,h}\|_{L^\infty((t_0-h,t_0+h))}\le Ch\omega_{t,1}^*(h).
\end{align}
By combining this and \eqref{eq:u-conv-diff-est}, we get
$$
\|u-P_{t_0,h}\|_{L^\infty((t_0-h,t_0+h))}\le Ch\omega_{t,1}^*(h).
$$
This implies that $\partial_tu$ is continuous in time with the modulus of continuity bounded by $C\int_0^r\frac{\omega_{t,1}^*(s)}sds$. A simple computation, along with Lemma~\ref{lem:Dini-sum-est}, shows 
$$
\int_0^r\frac{\omega_{t,1}^*(s)}sds\lesssim\int_0^r\frac{\omega_{t,1}(s)}sds\lesssim\sigma_1(r^{1/2}).
$$

To show that $u_t$ is continuous in $x'$, we let $x=(x',x_n)$ and $y=(y',x_n)$ with $h:=|x'-y'|>0$. Then
\begin{align}\label{eq:u_t-conti-x'}\begin{split}
    &|u_t(t,x',x_n)-u_t(t,y',x_n)|\\
    &\le \left|\frac{u(t+h^2,x)-u(t,x)}{h^2}-u_t(t,x)\right|+\left|\frac{u(t+h^2,y)-u(t,y)}{h^2}-u_t(t,y)\right|\\
    &\qquad+\frac{|u(t+h^2,x)-u(t+h^2,y)-u(t,x)+u(t,y)|}{h^2}\\
    &\lesssim\sigma_1(h)+\int_0^1\frac{|D_{x'}u(t+h^2,(1-s)x+sy)-D_{x'}u(t,(1-s)x+sy)|}hds\\
    &\lesssim \sigma_1(h),
\end{split}\end{align}
where we used that $\partial_tu$ is continuous in $t$ with the modulus of continuity bounded by $\sigma_1(r^{1/2})$ in the second step, and applied Theorem~\ref{thm:sol-time-est} to $D_{x'}u$ and used $\omega_{t,1}(r)\lesssim\sigma_1(r^{1/2})$ in the last step.

\medskip\noindent\emph{Step 3.} In this step we prove $D_{nn}u\in \mathring{C}$ and \eqref{eq:D_nu-condition}. Thanks to the results in Steps 1 and 2, $D_{nn}u\in \mathring{C}$ can be obtained by following the argument in Step 4 in the proof of Theorem~\ref{thm:holder-HO-reg-2}. Thus, it remains to prove \eqref{eq:D_nu-condition}. For any $0<h<1/4$, by applying Theorem~\ref{thm:DMO-reg} to $u^h$, we have
$$
|D_nu^h(t+h)-D_nu^h(t)|\le \omega_{x,1}(h^{1/2}).
$$
We take $j\in\mathbb{N}$ such that $\frac{h^{-\frac{1+\be}2}}2\le 2^j<h^{-\frac{1+\be}2}$. By using the previous estimate and the equality \eqref{eq:diff-quotient} applied to $D_nu$, we get
\begin{align*}
    &\frac{|D_nu(t+h)-D_nu(t)|}{h^{1/2}}\\
    &\le \frac{|D_nu(t+2^jh)|+|D_nu(t)|}{2^jh^{1/2}}+\sum_{i=1}^j\frac{|D_nu(t+2^ih)-2D_nu(t+2^{i-1}h)+D_nu(t)|}{2^{i}h^{1/2}}\\
    &\le 4\|D_nu\|_\infty h^{\be/2}+\sum_{i=1}^j\frac{|D_nu^{2^{i-1}h}(t+2^ih)-D_nu^{2^{i-1}h}(t+2^{i-1}h)|}{2^{\frac{i+1}2}}\\
    &\lesssim \|\D u\|_\infty h^{\be/2}+\sum_{i=1}^j\frac{\omega_{x,1}(2^{i/2}h^{1/2})}{2^{i/2}}\lesssim\|\D u\|_\infty h^{\be/2}+\int_{2^{1/2}}^{h^{-\frac{1+\be}4}}\frac{\omega_{x,1}(\rho h^{1/2})}{\rho^2}d\rho.
\end{align*}
To estimate the last term, we let
\begin{align*}
    \eta^\star_1(r):=\left(\|u\|_{1,\mu}+\|\bg\|_\infty+\int_0^1\frac{\eta_\bg(\tau)}\tau d\tau\right)(\eta_{A,1}(r)+r^{\be'})+\eta_{\bg,1}(r),
\end{align*}
and apply Lemma~\ref{lem:Dini-sum-est} and Fubini's theorem to have
\begin{align*}
    &\int_{2^{1/2}}^{h^{-\frac{1+\be}4}}\frac{\omega_{x,1}(\rho h^{1/2})}{\rho^2}d\rho\\
    &\lesssim \int_1^{h^{-\frac{1+\be}4}}\frac1{\rho^2}\left(\int_0^{\rho h^{1/2}}\frac{\eta_1^\star(s)}sds+\rho^\be h^{\be/2}\int_{\rho h^{1/2}}^1\frac{\eta_1^\star(s)}{s^{1+\be}}ds\right)d\rho\\
    &\lesssim \left(\int_0^{h^{1/2}}\int_1^{h^{-\frac{1+\be}4}}\frac{\eta_1^\star(s)}{\rho^2s}d\rho ds+\int_{h^{1/2}}^1\int_{h^{-1/2}s}^{h^{-\frac{1+\be}4}}\frac{\eta_1^\star(s)}{\rho^2 s}d\rho ds\right)\\
    &\quad +\int_{h^{1/2}}^1\int_1^{\infty}\frac{h^{\be/2}\eta_1^\star(s)}{\rho^{2-\be}s^{1+\be}}d\rho ds\\
&\lesssim\int_0^{h^{1/2}}\frac{\eta_1^\star(s)}sds+h^{1/2}\int_{h^{1/2}}^1\frac{\eta_1^\star(s)}{s^2}ds+h^{\be/2}\int_{h^{1/2}}^1\frac{\eta_1^\star(s)}{s^{1+\be}}ds.
\end{align*}
Here, we have
$$
h^{1/2}\int_{h^{1/2}}^1\frac{\eta_1^\star(s)}{s^2}ds\le h^{1/2}\int_{h^{1/2}}^1\frac{\eta_1^\star(s)}{s^{1+\be}}\left(\frac1{h^{1/2}}\right)^{1-\be}ds=h^{\be/2}\int_{h^{1/2}}^1\frac{\eta_1^\star(s)}{s^{1+\be}}ds.
$$
By combining the above three estimates, we get
$$
\frac{|\delta_{t,h}D_nu|}{h^{1/2}}\lesssim \sigma_1(h^{1/2})\quad\text{uniformly in }(t,x).
$$
By using
this estimate and the previous result $D_{x'}u\in \mathring{C}^{1/2,1}$ and $D_{nn}u\in \mathring{C}$, which implies $D_nD_x'u, D_{nn}u\in C_t$, we can follow the argument in Theorem~\ref{thm:diff-quo-conti} to conclude $\frac{\delta_{t,h}D_nu}{h^{1/2}}\in C_{X}(\overline{Q_1^+})$ uniformly in $h$.
\end{proof}

Before closing this section, we prove Remark~\ref{rem:DMO-HO-reg-par} and improve the result in Theorem~\ref{thm:DMO-HO-reg-par}.

\begin{proof}[Proof of Remark~\ref{rem:DMO-HO-reg-par}]
We only prove it for the cases $k=1$ and $k=2$, as the result for $k\ge3$ follows by induction. For $k=1$, we need to show $D_{x'}u$ is continuous in $x_n$. This can be simply inferred from \eqref{eq:mod-conti-space-est}.

When $k=2$, we need to show that $\partial_tu$ and $DD_\g u$ are continuous in $x_n$, $1\le\g\le n-1$. Indeed, by using $D_nu\in \mathring{C}^{1/2,1}$, we can follow the argument as in obtaining $\partial_tu\in C_{x'}$ (see \eqref{eq:u_t-conti-x'}) to derive $\partial_tu\in C_{x_n}$. To prove $DD_\g u\in C_{x_n}$, we use $D^2u\in C_{x_\g}$ to have for any $0<h<1/4$
\begin{align*}
    &\frac{|\delta_{x_\g,h}Du(t,x)-\delta_{x_\g,h}Du(t,x-he_n)|}{h}=\frac{|\de_{x_n,h}Du(t,x)-\delta_{x_n,h}Du(t,x-he_\g)|}{h}\\
    &\le \int_0^1|DD_nu(t,x-\tau he_n)-DD_nu(t,x-he_\g-\tau he_n)|d\tau\lesssim \sigma_1(h).
\end{align*}
By using this estimate, together with the triangle inequality and $DD_\g u\in C_{x_\g}$, we conclude
\begin{equation*}
    |D D_\g u(t,x)-D D_\g u(t,x-he_n)|\lesssim \sigma_1(h).\qedhere
\end{equation*}
\end{proof}


\section{Boundary Harnack Principle}\label{sec:BHP}

In this section, we establish the parabolic boundary Harnack principle for Hölder continuous coefficients (Theorem~\ref{thm:par-BHP}) by utilizing the results derived in the previous sections.

\begin{proof}[Proof of Theorem~\ref{thm:par-BHP}]
To begin with, by the classical boundary Schauder estimate, we have that $u,v\in C^{\frac{k+\be}2,k+\be}(\overline{Q^+_{3/4}})$, which gives
$$
\frac{u(X)}{x_n}=\frac1{x_n}\int_0^{x_n}\partial_nu(X',\rho)d\rho=\int_0^1\partial_nu(X',x_n\tau)d\tau\in C^{\frac{k-1+\be}2,k-1+\be}(\overline{Q^+_{3/4}}).
$$
Similarly, we have $v/x_n\in C^{\frac{k-1+\be}2,k-1+\be}$. Moreover, the condition $\partial_{n}u>0$ on $Q_1'$ implies $u/x_n\ge c>0$ in $Q_{3/4}^+$, and hence $w:=v/u$ satisfies
$$
\|w\|_{L^\infty(Q^+_{3/4})}\le \|v/x_n\|_{L^\infty(Q^+_{3/4})}\|(u/x_n)^{-1}\|_{L^\infty(Q_{3/4}^+)}\le C.
$$
We observe that by the symmetry of $A$,
\begin{align*}
    \div(u^2A\D w)&=\div(uA\D v-vA\D u)=(u\partial_tv-v\partial_tu)+(ug-vf)\\
    &=u^2\partial_tw+(ug-vf).
\end{align*}
This gives
\begin{align*}
    \div(x_n^2A\D w)&=\div((u/x_n)^{-2}u^2A\D w)\\
    &=(u/x_n)^{-2}\div(u^2A\D w)+\mean{\D((u/x_n)^{-2}),u^2A\D w}\\
    &=x_n^2\partial_tw+(u/x_n)^{-2}(ug-vf)+2x_n\mean{\vec e_n-(u/x_n)^{-1}\D u,A\D w}\\
    &=x_n^2\partial_tw+x_n\tilde g,
\end{align*}
where
\begin{align}
    \label{eq:rhs-w}
    \tilde g:=(u/x_n)^{-1}g-(u/x_n)^{-2}(v/x_n)f+\mean{2A(\vec e_n-(u/x_n)^{-1}\D u),\D w}.
\end{align}
We write $\tilde g=\tilde g_1+\mean{\bb,\D w}$, where $\tilde g_1=(u/x_n)^{-1}g-(u/x_n)^{-2}(v/x_n)f$ and $\bb=2A\left(\vec e_n-(u/x_n)^{-1}\D u\right)$. We further have
\begin{align}\label{eq:pde-w}
\div(x_n^2A\D w)-x_n^2\partial_tw=\div(x_n^2\tilde{\bg}),
\end{align}
where 
$$
\tilde\bg(X',x_n):=\frac{\vec e_n}{x_n^2}\int_0^{x_n}\rho \tilde g(X',\rho)d\rho=\vec e_n\int_0^1\tau \tilde g(X',\tau x_n)d\tau.
$$
We now split the remainder of the proof into three steps.

\medskip\noindent\emph{Step 1.} In this step, we establish the theorem when $k=1$. As in Section~\ref{subsec:reg-space}, we will derive an a priori estimate of the Hölder continuity of $\D w$ under the assumption $\D w\in C^{\be/2,\be}(\overline{Q_{3/4}^+})$. The general case follows from a standard approximation argument. Note that $[\tilde\bg]_{C^{\be/2,\be}(\overline{Q_r^+(X_0)})}\le [\tilde g]_{C^{\be/2,\be}(\overline{Q_r^+(X_0)})}<\infty$ whenever $X_0\in Q'_1$ and $Q_r^+(X_0)\subset Q^+_{3/4}$. Then, by following the arguments in Lemmas~\ref{lem:psi-bdry-est} - \ref{lem:psi-est} (see also \cite{Don12}*{Section 5}, \cite{DonEscKim21}*{Section 3.2}) with 
$$
\tilde\psi(X_0,r):=\left(\dashint_{Q_r^+(X_0)}|\D w-\mean{\D w}_{Q_r^+(X_0)}|^2d\tilde\mu\right)^{1/2},\quad X_0\in Q^+_{3/4},\,\, 0<r<1/4,
$$
where $d\tilde\mu=x_n^2 dX$, instead of $\psi$, we obtain an analogue of \eqref{eq:psi-est}: for every $X_0\in Q^+_{3/4}$ and $0<\rho<r<1/4$
\begin{align*}
    \tilde\psi(X_0,\rho)&\le C(\rho/r)^\be\|\D w\|_{L^\infty(Q^+_{5r}(X_0))}\\
    &\qquad+C\left(\|\D w\|_{L^\infty(Q_{8r}^+(X_0))}[A]_{C^{\be/2,\be}(\overline{Q_{8r}^+(X_0)})}+[\tilde\bg]_{C^{\be/2,\be}(\overline{Q^+_{8r}(X_0)})}\right)\rho^\be,
\end{align*}
where $C=C(n,\la,\be)>0$. This implies that for any $0<r_0<R_0<3/4$, $X\in Q^+_{r_0}$ and $0<\rho<\frac{R_0-r_0}{16}$,
$$
\tilde\psi(X,\rho)\le CM\rho^\be,\quad M:=\frac{\|\D w\|_{L^\infty(Q_{R_0}^+)}}{(R_0-r_0)^\be}+[\tilde g]_{C^{\be/2,\be}(\overline{Q^+_{R_0})}}
$$
for some constant $C=C(n,\la,\be,[A]_{C^{\be/2,\be}(\overline{Q^+_{3/4}})})>0$. By the standard Campanato space embedding (see e.g. the proof of Theorem~3.1 in \cite{HanLin97}), we have
\begin{align}
    \label{eq:grad-w-holder}\begin{split}
    [\D w]_{C^{\be/2,\be}(\overline{Q_{r_0}^+})}&\le C\left(M+\frac{\|\D w\|_{L^\infty(Q^+_{R_0})}}{(R_0-r_0)^\be}\right)\\
    &\le C\left(\frac{\|\D w\|_{L^\infty(Q^+_{R_0})}}{(R_0-r_0)^\be}+[\tilde g]_{C^{\be/2,\be}(\overline{Q^+_{R_0}})}\right).
\end{split}\end{align}
Recall \eqref{eq:rhs-w}. By using $u/x_n\ge c>0$ in $Q_{3/4}^+$, $u=0$ on $Q_1'$, $b\in C^{\be/2,\be}(\overline{Q_{3/4}^+})$, and $|(u/x_n-\partial_nu)(X',x_n)|\le \int_0^1|\partial_nu(X',x_n\tau)-\partial_nu(X',x_n)|d\tau\le Cx_n^\be$, we get
$$
[\tilde g]_{C^{\be/2,\be}(\overline{Q^+_{R_0}})}\le [\tilde g_1]_{C^{\be/2,\be}(\overline{Q_{3/4}^+})}+C\|\D w\|_{L^\infty(Q_{R_0}^+)}+CR_0^\be[\D w]_{C^{\be/2,\be}(\overline{Q_{R_0}^+})}.
$$
Combining this with \eqref{eq:grad-w-holder} and using the interpolation inequality yield that for every $0<r_0<R_0<3/4$
\begin{align*}
    [\D w]_{C^{\be/2,\be}(\overline{Q_{r_0}^+})}&\le C\left(\frac{\|\D w\|_{L^\infty(Q_{R_0}^+)}}{(R_0-r_0)^\be}+[\tilde g_1]_{C^{\be/2,\be}(\overline{Q_{3/4}^+})}\right)+CR_0^\be[\D w]_{C^{\be/2,\be}(\overline{Q_{R_0}^+})}\\
    &\le C\left(\frac{\|w\|_{L^\infty(Q_{3/4}^+)}}{(R_0-r_0)^{1+\be}}+[\tilde g_1]_{C^{\be/2,\be}(\overline{Q_{3/4}^+})}\right)+CR_0^\be[\D w]_{C^{\be/2,\be}(\overline{Q_{R_0}^+})}.
\end{align*}
We fix $S_0\in (0,3/4)$ so that $CS_0^\be<1/2$ and apply \cite{Gia83}*{Lemma~3.1 of Ch.V} to deduce
$$
[\D w]_{C^{\be/2,\be}(\overline{Q^+_{S_0/2}})}\le C\left(\|w\|_{L^\infty(Q_{3/4}^+)}+[\tilde g_1]_{C^{\be/2,\be}(\overline{Q_{3/4}^+})}\right).
$$
By varying the centers $X_0'\in Q'_1$, we have the similar estimate for $$[\D w]_{C^{\be/2,\be}(\overline{Q_{2/3}^+}\cap\{x_n<S_0/2\})}.$$

Moreover, since $w$ satisfies 
$$
\div(x_n^2A\D w)-x_n^2\partial_tw=x_n\tilde g_1+\mean{x_n\bb,\D w}
$$ 
in $\cQ:=\overline{Q_{2/3}^+}\cap\{x_n>S_0/4\}$ and $0<c<x_n^2<C$ in $\cQ$, we can estimate $\|\D w\|_{C^{\be/2,\be}(\cQ)}$ by following the argument in \cite{Don12}*{Theorem~2}. Therefore, $\D w\in C^{\be/2,\be}(\overline{Q^+_{2/3}})$. Then, $w$ can be seen as a solution of \eqref{eq:pde-w} in $Q_{2/3}^+$ with $\tilde\bg\in C^{\be/2,\be}(\overline{Q_{2/3}^+})$. By Theorem~\ref{thm:holder-HO-reg}, we conclude $w\in C^{\frac{1+\be}2,1+\be}(\overline{Q_{1/2}^+})$.

\medskip\noindent\emph{Step 2.} In this step, we prove Theorem~\ref{thm:par-BHP} for $k=2$. It is enough to show $\D_{x'}w\in C^{\frac{1+\be}2,1+\be}$, $D_nw\in C^{\frac{1+\be}2}_t$, $\partial_tw\in C^{\be/2,\be}$ and $D_{nn}w\in C^{\be/2,\be}$. From \eqref{eq:pde-w}, we see that $\D_{x'}w$ satisfies
$$
\div(x_n^2 A\D(\D_{x'}w))-x_n^2\partial_t(\D_{x'}w)=\div(x_n^2(\D_{x'}\tilde\bg-\D_{x'}A\D w)).
$$
Clearly, $\D_{x'}A,\D w\in C^{\be/2,\be}$. Moreover, a direct computation gives
\begin{align*}
    \D_{x'}\tilde\bg=\vec{e}_n\int_0^1\tau(\D_{x'}\tilde g_1+\mean{\D_{x'}\bb,\D w}+\mean{\bb,\D(\D_{x'}w)})(X',\tau x_n)d\tau.
\end{align*}
Since $\D_{x'}\tilde g_1+\mean{\D_{x'}\bb,\D w}\in C^{\be/2,\be}$, we obtain by Step 1 that $\D_{x'}w\in C^{\frac{1+\be}2,1+\be}$.

Next, we observe that $w^h(t,x)=\frac{w(t,x)-w(t-h,x)}{h^{1/2}}$ solves
$$
\div(x_n^2 A\D w^h)-x_n^2\partial_tw^h=\div(x_n^2\hat\bg^h)\quad\text{in }Q_1^+,
$$
where $\hat\bg^h(X)=\tilde\bg^h(X)-A^h(X)\D w(t-h,x)$. A direct calculation shows
$$
\tilde\bg^h(X)=\vec{e}_n\int_0^1\tau((\mean{\bb,\D w^h}+\tilde g_1^h)(X',\tau x_n)+\mean{\bb^h(X',\tau x_n),\D w(t-h,x',\tau x_n)})d\tau.
$$
Since $A^h, \D w, \tilde g_1^h, \bb^h\in C^{\be/2,\be}$, we have by the result in Step 1 that $w^h\in C^{\frac{1+\be}2,1+\be}$. This is an analogue of \eqref{eq:u-h-holder}, and we saw in Step 3 in the proof of Theorem~\ref{thm:holder-HO-reg-2} that $w^h\in C^{\frac{1+\be}2,1+\be}$ implies $D_nw\in C^{\frac{1+\be}2}_t$ and $\partial_tw\in C^{\be/2,\be}$. 

It remains to show $D_{nn}w\in C^{\be/2,\be}$. Recall that 
$$
\div(x_n^2A\D w)=x_n^2\partial_tw+x_n\tilde g=x_n^2\partial_tw+x_n\tilde g_1+x_n\mean{\bb,\D w}.
$$ 
Then, for $W:=\mean{A\D w,\vec{e}_n}-\frac{\tilde g_1}2\in C^{\be/2,\be}(\overline{Q_{1/2}^+})$, a direct computation gives
$$
x_n^{-2}\partial_n(x_n^2W)=-\sum_{i=1}^{n-1}\partial_i(\mean{A\D w,\vec{e}_i})-\frac{\partial_n\tilde g_1}2+\partial_tw+\mean{\bb/x_n,\D w}=:\bar f.
$$
It follows that 
$$
W=\frac1{x_n^2}\int_0^{x_n}\rho^2\bar f(X',\rho)d\rho,
$$
which yields 
$$
D_nW=\bar f-2\int_0^1\rho^2\bar f(X',x_n\rho)d\rho.
$$
From
$$
\frac{\D_{x'}u(X',x_n)}{x_n}=\int_0^1D_n\D_{x'}u(X',x_n\rho)d\rho\in C^{\be/2,\be}
$$ 
and 
$$
\frac1{x_n}(u(X)/{x_n}-D_nu(x))=-\int_0^1\rho D_{nn}u(X',x_n\rho)d\rho\in C^{\be/2,\be},
$$ 
we find $\bb/x_n\in C^{\be/2,\be}$, and thus $\bar f\in C^{\be/2,\be}$. Then $D_nW\in C^{\be/2,\be}$, and hence $D_{nn}w\in C^{\be/2,\be}$.

\medskip\noindent\emph{Step 3.} Once we have Theorem~\ref{thm:par-BHP} for $k=1,2$, we can use induction as in the proof of Theorem~\ref{thm:holder-HO-reg} to achieve Theorem~\ref{thm:par-BHP} for $k\ge3$.
\end{proof}


\appendix
\section{Properties of some Dini functions}

\begin{lemma}
    \label{lem:Dini-sum-est}
    Let $\sigma:[0,1]\to[0,\infty)$ be a Dini function, $0<\ka<1$ and $0<\be<1$. For $\be'=\frac{1+\be}2$, we set
    $$
    \tilde\sigma(r):=\sum_{j=0}^\infty\ka^{\be' j}\sigma(\ka^{-j}r)[\ka^{-j}r\le1],\quad \hat\sigma(r):=\sup_{\rho\in[r,1]}(r/\rho)^{\be'}\tilde\sigma(r).
    $$
    Then
    \begin{align*}
            \int_0^r\frac{\hat\sigma(\rho)}\rho d\rho\lesssim \int_0^r\frac{\sigma(\rho)}\rho d\rho+r^\be\int_r^1\frac{\sigma(\rho)}{\rho^{1+\be}}d\rho.
    \end{align*}
\end{lemma}

\begin{proof}
By using \cite{DonEscKim18}*{(2.30)}, we have
\begin{align*}
    \hat\sigma(r)&\lesssim\sum_{j=0}^\infty\ka^{\be'j}\tilde\sigma(\ka^{-j}r)[\ka^{-j}r\le1]\lesssim \sum_{j=0}^\infty\ka^{\be'j}\sum_{i=0}^\infty\ka^{\be'i}\sigma(\ka^{-(j+i)}r)[\ka^{-(j+i)}r\le1]\\
    &\lesssim\sum_{j=0}^\infty j\ka^{\be'j}\sigma(\ka^{-j}r)[\ka^{-j}r\le1]\lesssim\sum_{j=0}^\infty\ka^{\be j}\sigma(\ka^{-j}r)[\ka^{-j}r\le1].
\end{align*}
It follows that
\begin{align*}
    &\int_0^r\frac{\hat\sigma(\rho)}\rho d\rho=\sum_{j=0}^\infty\ka^{\be j}\int_0^{\ka^{-j}r}\frac{\sigma(\rho)}\rho[\rho\le1]d\rho\\
    &=\sum_{j=0}^\infty\ka^{\be j}\int_0^r\frac{\sigma(\rho)}\rho d\rho+\sum_{j=0}^\infty\ka^{\be j}\sum_{i=0}^{j-1}\int_{\ka^{-i}r}^{\ka^{-(k+1)}r}\frac{\sigma(\rho)}\rho[\rho\le1]d\rho\\
    &\lesssim \int_0^r\frac{\sigma(\rho)}\rho d\rho+\sum_{i=0}^\infty\sum_{j=i}^\infty\ka^{\be j}\int_{\ka^{-i}r}^{\ka^{-(i+1)}r}\frac{\sigma(\rho)}\rho[\rho\le1]d\rho\\
    &\lesssim \int_0^r\frac{\sigma(\rho)}\rho d\rho+\sum_{i=0}^\infty \ka^{\be i}\int_{\ka^{-i}r}^{\ka^{-(i+1)}r}\frac{\sigma(\rho)}\rho[\rho\le1]d\rho\\
    &\lesssim \int_0^r\frac{\sigma(\rho)}\rho d\rho+\sum_{i=0}^\infty r^{\be}\int_{\ka^{-i}r}^{\ka^{-(i+1)}r}\frac{\sigma(\rho)}{\rho^{1+\be}}[\rho\le1]d\rho\lesssim \int_0^r\frac{\sigma(\rho)}\rho d\rho+r^{\be}\int_r^1\frac{\sigma(\rho)}{\rho^{1+\be}}d\rho.\qedhere
\end{align*}
\end{proof}


\end{document}